\documentclass[11pt,reqno]{amsart} 
\usepackage{amsfonts, amsmath, amssymb, amsthm, color}
\usepackage[margin=0.97in]{geometry}
\usepackage{esint}

\newtheorem{thm}{Theorem}[section]

\newtheorem{prop}[thm]{Proposition}
\newtheorem{lemma}[thm]{Lemma}

\newtheorem{remark}[thm]{Remark}
\theoremstyle{definition}

\newcommand{\N}{\mathbb{N}}
\newcommand{\R}{\mathbb{R}}

\newcommand{\dist}{\textnormal{dist}}

\allowdisplaybreaks
\numberwithin{equation}{section}

\makeatletter
\@namedef{subjclassname@2020}{%
\textup{2020} Mathematics Subject Classification}
\makeatother

\begin{document}
\title[Sinh-Gordon flow]{Critical sinh-Gordon flow with non-negative \\weight functions}

\author{Qiang Fei}
\address[Qiang Fei]{School of Mathematics and Statistics, Central South University,
		Changsha 410083, Hunan, People's Republic of China}
\email{math\_qiangfei@163.com}

\author{Aleks Jevnikar}
 \address[Aleks Jevnikar]{Department of Mathematics, Computer Science and Physics, University of Udine, Via delle Scienze 206, 33100 Udine, Italy}
\email{aleks.jevnikar@uniud.it}

\author{Sang-Hyuck Moon}
\address[Sang-Hyuck Moon]{Sang-Hyuck Moon, Department of Mathematics and Institute of Mathematical Science, Pusan National University, Busan 46241, Republic of Korea}
\email{shmoon@pusan.ac.kr}

\thanks{A.J. is partially supported by INdAM-GNAMPA project ``{\em Analisi qualitativa di problemi differenziali non lineari}'' and PRIN Project 20227HX33Z ``{\em Pattern formation in nonlinear phenomena}'' and is a member of the INDAM Research Group ``Gruppo Nazionale per l'Analisi Matematica, la Probabilità e le loro Applicazioni''. S.-H Moon was supported by the National Research Foundation of Korea (NRF) grant funded by the Korea government(MSIT) (No. 2022R1C1C200982, No. 2022R1A2C4002585) and a New Faculty Research Grant of Pusan National University, 2025. }

\begin{abstract}
The aim of this article is twofold: one one side we introduce and study the properties of a critical sinh-Gordon type flow
\begin{equation*}
{\frac{\partial}{\partial t}}e^u=\Delta_gu+8\pi\left({\frac{h_1e^u}{\int_{\Sigma}h_1e^udV_g}}-1\right)-\rho_2\left({\frac{h_2e^{-u}}{\int_{\Sigma}h_2e^{-u}dV_g}}-1\right),
\end{equation*}
where $\rho_2<8\pi$, $h_1,h_2$ are non-negative weight functions and $\Sigma$ is a closed Riemannian surface. Secondly, under suitable geometric conditions, we prove the convergence of the flow to a solution of the critical sinh-Gordon equation, extending the result of Zhou (2008) to the case of non-negative weights. The argument is based on a careful blow-up analysis. Some remarks about a Toda flow are also given.
\end{abstract}


\maketitle


\section{Introduction}
Let $(\Sigma,g)$ be a closed Riemann surface with metric $g$ and let $h_1, h_2$ be smooth non-negative functions.
For simplicity, we will assume that the area $|\Sigma|_g$ of the surface equals 1 throughout the paper.
We are concerned with the following sinh-Gordon equation
\begin{equation}\label{maineq-s}
-\Delta_gu=\rho_1\left(\frac{h_1e^u}{\int_{\Sigma}h_1e^udV_g}-1\right)-\rho_2\left(\frac{h_2e^{-u}}{\int_{\Sigma}h_2e^{-u}dV_g}-1\right) \quad \mbox{on } \Sigma,
\end{equation}
where $\rho_1$ and $\rho_2$ are non-negative constants.

Derived from Onsager's vortex model\,\cite{Ons}, equation\,\eqref{maineq-s} appears in\,\cite{JM,PL} as a model in the description of the mean field of the equilibrium turbulence with arbitrarily signed vortices  from different statistical arguments, and for more physical background concerning 2D-turbulence, see\,\cite{Chorin, Lions, Newton} and the references therein. In addition to turbulent Euler flows, it also arises as a mean field equation in the description of self-dual condensates of some Chern-Simon-Higgs model, see \cite{Caffarelli, DJLW-C, DJLW-M, MSGT, Tarantello}. As for conformal geometry, if $\rho_1=8\pi$ and $\rho_2=0,$ \eqref{maineq-s} is related to the well-known Kazdan-Warner problem of prescribing the Gaussian curvature, see \cite{KW, CGY, Cheng-Lin} and the references therein.
\par

The sinh-Gordon equation \eqref{maineq-s} has a variational structure, and its solutions correspond to the critical points of the functional $I_{\rho_1,\rho_2}: H^1(\Sigma) \to \R$
\[
I_{\rho_1,\rho_2}(u)={\frac{1}{2}}\int_{\Sigma}|\nabla_g u|^2dV_g-\rho_1\log\int_{\Sigma}h_1e^{u}dV_g-\rho_2\log\int_{\Sigma}h_2e^{-u}dV_g +\big(\rho_1-\rho_2\big)\int_{\Sigma}u dV_g.    \]
One fundamental tool to deal with this kind of functionals is the Moser-Trudinger type inequality
\begin{equation}\label{MTineq}
\log\int_{\Sigma}e^{u-\overline{u}}dV_g+\log\int_{\Sigma}e^{-u+\overline{u}}dV_g\leq{\frac{1}{16\pi}}\int_{\Sigma}|\nabla_gu|^2dV_g+C, \quad \forall u \in H^1(\Sigma),
\end{equation}
where $\overline{u}$ represents the average of $u$.
In the subcritical case, i.e., when $\rho_1,\rho_2 \in(0,8\pi)$, by the Moser-Trudinger inequality \eqref{MTineq}, $I_{\rho_1,\rho_2}$ is bounded from below and coercive. Thus, the global minima of $I_{\rho_1,\rho_2}$ can be attained by the direct minimization.
However, in the critical case, i.e., when $\rho_i \le 8\pi$, $i=1,2$ and $\max(\rho_1,\rho_2)=8\pi$, the functional is bounded from below but not coercive.
This leads to a loss of compactness and makes the existence problem quite subtle and existence of solutions typically depends on the geometry of the underlying surface. This is why the literature about this case is very limited, as we will comment later on. The goal of this paper is to introduce a new tool to address this problem and to extend some previous results.

\medskip

For the supercritical case, i.e., when $\max(\rho_1,\rho_2)>8\pi$,  $I_{\rho_1,\rho_2}$ is unbounded from below and direct minimization can not be applied to the problem. This was considered by many authors, especially for $\rho_2=0$, which reduces to the well-known mean-field equation
\begin{equation}\label{Liou-e}
-\Delta_gu=\rho\left({\frac{he^u}{\int_{\Sigma}he^udV_g}}-1\right).
\end{equation}
Indeed, many techniques have been developed like degree counting and min-max schemes.
For example, we refer to the papers \cite{Chen-Lin-T,DJLW-E,Djadli,Lin,MR,Malchiodi} and the references therein.
On the other hand, there are few results about the sinh-Gordon equation \eqref{maineq-s} in the supercritical case.
We refer to the papers \cite{BJMR, AJ, JWYZ, Zhou-S} and references therein.
We also remark that the problem has some analogies with the Toda system, see \cite{BJMR}.

From now on, we will focus on the sinh-Gordon equation and the mean field equation in the critical case.
For the mean field equation \eqref{Liou-e} with $\rho=8\pi$ and a positive function $h \in C^\infty(\Sigma)$, Ding, Jost, Li and Wang (see \cite{DJLW}) proved the first existence result under a geometric condition.
They considered the minimizer $u_\epsilon$ of the slightly subcritical case
\[
I_{8\pi-\epsilon}(u)=\frac12\int_{M}|\nabla_gu|^2dV_g+(8\pi-\epsilon)\int_{\Sigma}u dV_g-(8\pi-\epsilon)\log\int_{\Sigma}he^udV_g.
\]
When $u_{\epsilon}$ blows up and does not converges in $H^1(\Sigma)$, by blow-up analysis, they inferred the following lower bound related to the geometry of $\Sigma$
\begin{equation}\label{lowerbd1}
\inf\limits_{H^1(M)}I_{8\pi}(u) \geq -8\pi-8\pi\log\pi-4\pi\max_{x_0\in M}\big(A(x_0)+2\log h(x_0)\big).
\end{equation}
On the other hand, they also constructed a test function $\phi_{\epsilon}$ such that, for small $\epsilon>0$, $I_{8\pi}(\phi_{\epsilon})$ is strictly less than the right hand side of \eqref{lowerbd1}, which contradicts the blow-up property. Consequently, $u_{\epsilon}$ converges in $H^1(\Sigma)$ to the solution $\tilde{u}$ of \eqref{Liou-e} with $\rho=8\pi$.
Later, Yang and Zhu (see \cite{YangZhu}) generalized the above result for a non-negative function $h$ by excluding the possibility of the blow-up at zeros of $h(x)$ based on compactness-concentration lemma in \cite{DJLW-C}.
Recently, in \cite{SunZhu-E,Zhu}, the authors proved that such existence results still hold even when $h$ is sign-changing.

There are also existence results for the sinh-Gordon equation in the critical case.
In \cite{Zhou}, Zhou obtain the existence result of \eqref{maineq-s} with $h_1,h_2 \equiv 1$, $\rho_1=8\pi,\rho_2\in(0,8\pi]$ under some geometric conditions generalizing \cite{DJLW}.
The argument is in the spirit of Toda systems (see \cite{JLW}) and exploits the compact-concentration theorem established by Ohtsuka-Suzuki \cite{OS}.  However, when $h_1,h_2$ are non-negative functions, we can not directly follow this approach.
Thus, the main goal of this paper is to provide an alternative proof of the previous results and to extend them to non-negative functions $h_1, h_2$.

We will base our analysis on the flow method, which was already exploited for the mean field equation \eqref{Liou-e}.
In \cite{Cas1,Cas2}, the author introduced the following mean field type flow
\begin{equation}\label{Cas-f}
\frac{\partial}{\partial t}e^v=\Delta v-Q+\rho{\frac{e^v}{\int_\Sigma e^vdV_g}}, \quad \quad
v(\cdot,0)=v_0(x) \in C^{2+\alpha}(\Sigma),
\end{equation}
$\alpha\in(0,1)$,where $Q\in C^{\infty}(\Sigma)$ is a given function such that  $\int_\Sigma QdV_g=\rho$.
We note that the time-independent solution satisfies a mean-field type equation, which is equivalent to \eqref{Liou-e}.
In \cite{Cas2}, Cast\'eras proved the global existence of the solution.
Moreover, using the compactness theorem in \cite{Cas1}, he proved the convergence of the flow $v(t)$ to a solution of the mean field equation associated to \eqref{Cas-f} provided that $\rho\neq 8N\pi$ for $N\in\mathbb{N}^*$.
However, such compactness theorem fails in the critical case $\rho=8\pi$.
In \cite{LiZhu}, the authors used the idea of \cite{DJLW} to overcome this difficulty,
proving a lower bound of the corresponding functional $I_{8\pi}$ when the flow is not bounded. They constructed a test function $\phi_\epsilon$ such that $I_{8\pi}(\phi_\epsilon)$ is smaller than the lower bound under the geometric condition in \cite{DJLW}.
This means the flow converges when we choose an appropriate initial data.
Subsequently, in \cite{SunZhu}, Sun and Zhu generalized this approach for non-negative functions $h$.
Finally, in \cite{LiXu}, Li and Xu generalized the above result to the case of the sign-changing function $h$.
Yang and Wang (see \cite{WY}) also considered the mean field type flow with a sign-changing function $h$ when a finite isometric group acts on the surface and $h$ is invariant under the group action.
For other variant of mean field type flows, see \cite{Lin-Yang,Yu-Yang,Zhang-Yang}. For $Q$-curvature flows we refer to the recent result \cite{q-curv} and the references therein. Recently, there is also a result \cite{ParkZhang} about a parabolic system related to Toda systems.

\medskip

Motivated by \cite{LiXu,LiZhu,SunZhu,WY}, we introduce the following evolution problem \eqref{maineq} to deal with the critical  sinh-Gordon equation \eqref{maineq-s}
\begin{equation}\label{maineq}
\left\{\begin{aligned}
&{\frac{\partial}{\partial t}}e^u=\Delta_gu+8\pi\left({\frac{h_1e^u}{\int_{\Sigma}h_1e^udV_g}}-1\right)-\rho_2\left({\frac{h_2e^{-u}}{\int_{\Sigma}h_2e^{-u}dV_g}}-1\right),\\
&u(\cdot,0)=u_0 \in C^{2+\alpha}(\Sigma),
\end{aligned}\right.
\end{equation}
$\alpha\in(0,1)$. We note that it is a gradient flow with respect to the functional $I_{8\pi,\rho_2}$. For simplicity, we will denote it $J_{\rho_2}$
\begin{equation}\label{main-EF}
J_{\rho_2}(u)={\frac{1}{2}}\int_{\Sigma}|\nabla_g u|^2dV_g-8\pi\log\int_{\Sigma}h_1e^{u}dV_g-\rho_2\log\int_{\Sigma}h_2e^{-u}dV_g +\big(8\pi-\rho_2\big)\int_{\Sigma}u dV_g.
\end{equation}


To guarantee some global properites of the flow, we always assume that $h_1 h_2 \not\equiv 0$, see the discussion later on. We investigate the properties of the latter flow and, in particular, prove the following result.

\begin{thm}\label{global}
Fix $\alpha \in (0,1)$. For any initial data $u_0 \in C^{2+\alpha}(\Sigma)$,  there exists a unique global solution $u\in C^{2+\alpha,1+\alpha/2}(\Sigma \times [0,+\infty))$ to \eqref{maineq}.
\end{thm} 
\par

In the second part of the paper we exploit the latter result to establish existence of solutions for the critical sinh-Gordon equation with non-negative weight functions, generalizing the results of Ding-Jost-Li-Wang \cite{DJLW} and Zhou \cite{Zhou} for the mean field and sinh-Gordon equations, respectively. Before stating the theorem, we introduce some notations.

\medskip

Let $K$ denote the Gaussian curvature of $\Sigma$. For each $p \in \Sigma$, let $G_p$ be the Green function satisfying
\begin{equation}\label{green}
-\Delta_g G_p=8\pi\delta_p -8\pi \quad \text{ on }  \Sigma, \quad \int_\Sigma G_p\, dV_g=0,
\end{equation}
and $A(p)$ be the regular part of the Green function. More precisely, 
$G_p(x)$ has the following expansion in normal coordinates near $p$:
\begin{equation}\label{Ap}
G_p(x)=-4\log \dist_g (x,p)+A(p)+O(r^2), \quad r=\dist_g(x,p).
\end{equation}

For $p\in\Sigma$, let $\Gamma_p$ be the set of solutions $w_p \in H^1(\Sigma)$ to the singular mean field equation
\begin{equation}\label{smfeq}
-\Delta_g w_p=\rho_2 \left(\frac{h_2 e^{-G_p}e^{w_p}}{\int_\Sigma h_2 e^{-G_p}e^{w_p}dV_g}-1\right) \quad \text{ on }  \Sigma, \quad \int_\Sigma w_p\, dV_g=0.
\end{equation}
We also introduce the functional
\begin{equation}\label{tildeJ}
\tilde{J}_{p}(u)={\frac{1}{2}}\int_{\Sigma}|\nabla_gu|^2dV_g-\rho_2\log\int_{\Sigma}h_2e^{-G_{p}}e^udV_g,\quad\ \forall  u\in H^1(\Sigma) \  \text{ with } \int_\Sigma u =0.
\end{equation}
whose critical points solve \eqref{smfeq}.

Now we are prepared to state the second main theorem.
\begin{thm}\label{main-thm}
Let $\rho_2 \in (0,8\pi)$ and $h_1, h_2$ be smooth non-negative functions.
Suppose 
$$
8\pi-\rho_2-2K(p_0)+\Delta_g\log h_1(p_0)>0
$$ 
for any minimizer $p_0 \in \Sigma$ of $p \mapsto \inf\limits_{p \in \Sigma}\inf\limits_{w \in \Gamma_{p}} (\tilde{J}_{p}(w)-4\pi A(p) -8\pi\log h_1(p))$.
Then, there exists an initial datum $u_0 \in C^{2+\alpha}(\Sigma)$ such that the flow $u(x,t)$ converges in $C^2(\Sigma)$ to a solution $u_\infty$ of the critical sinh-Gordon equation \eqref{maineq-s} with $\rho_1=8\pi$, $\rho_2 \in (0,8\pi)$.
\end{thm}

In fact, when $\rho_2=0$, $\Gamma_p=\{0\}$ and the condition is consistent with the one for the mean field equation. However, when $\rho_2>0$, due to the effect of $e^{-u}$ term in \eqref{maineq-s}, $\tilde{J}_{p}(w)$ appears in the lower bound (see Proposition \ref{lowerbound}).

\medskip

We sketch now the proofs highlighting the differences with the previous results.
We first prove that the solution $u(x,t)$ of \eqref{maineq} exists globally in time.
For this, we derive several a priori estimates (Proposition \ref{H1} -- Proposition \ref{C2alpha}).
However, this is different from previous works \cite{Cas2, SunZhu, WY} since we have to control the $e^{-u}$ term and $\int_\Sigma e^{-u}$ is not conserved. Even though we can perform the below explained blow up analysis for $\rho_2=8\pi$ and sign-changing weight functions, we use in this step $\rho_2<0$ and $h_1,h_2$ non-negative. We postpone to a future work the discussion of this point and possible extensions of this method.

Next, we prove that the time-slices $u(t_n)$ can not blow up at the zero set of $h_1$ (Proposition \ref{Selection}).
We also show $-u(t_n)$ does not blow up (Proposition \ref{Never-S}).
The blow-up analysis is delicate especially when  $u(t_n)$ and $-u(t_n)$ blow up at the same point.
We adapted the idea of the selection process in \cite{JWY,LWZ} and used the hypothesis $\rho_2 <8\pi$ subtly.
We remark that we can not apply the result in \cite{JWY} directly due to the time derivative term and the non-negativeness of $h_1, h_2$.

Based on this blow-up analysis, we prove the lower bound of the functional $\lim_{t \to\infty} J_{\rho_2}(u(t))$ when the time-slices $u(t_n)$ blow up (Proposition \ref{lowerbound}).
As we mentioned after the main theorem, $\tilde{J}_{p}(w)$ appears in the lower bound, since $-u(t_n)$ converges to a solution of \eqref{smfeq}.
Then we construct a test function $\tilde{\Phi}_\epsilon$ such that, under suitable geometric conditions, $J_{\rho_2}(\tilde{\Phi}_\epsilon)$ is smaller than the lower bound (Proposition \ref{TFthm}).
For this purpose, we have to choose a solution $w$ of \eqref{smfeq} achieving the infimum of the leading term, so we prove the compactness of the solution set of \eqref{smfeq} (Proposition \ref{Compt}).
Finally, we show the convergence of the flow using a priori estimates and {\L}ojasiewicz-Simon gradient inequality \cite{PM}. 

We conclude the introduction with the following remark about the flow method for Toda systems.

\begin{remark}
By the same method we can address following critical $SU(3)$ Toda system:
\begin{equation}\label{Toda}
\left\{\begin{aligned}
-\Delta_gu_1 =2\rho_1\left({\frac{h_1e^{u_1}}{\int_{\Sigma}h_1e^{u_1}dV_g}}-1\right)-\rho_2\left({\frac{h_2e^{u_2}}{\int_{\Sigma}h_2e^{u_2}dV_g}} -1\right),\\
-\Delta_gu_2 =2\rho_2\left({\frac{h_2e^{u_2}}{\int_{\Sigma}h_2e^{u_2}dV_g}}-1\right)-\rho_1\left({\frac{h_1e^{u_1}}{\int_{\Sigma}h_1e^{u_1}dV_g}}- 1\right),
\end{aligned}\right.
\end{equation}
where $\rho_1=4\pi,\rho_2\in(0,4\pi)$, $h_1,h_2 \in C^\infty(\Sigma)$, $h_1, h_2 \ge 0$. Indeed, following the same strategy as before, we are able to carry out the blow up analysis and the construction of suitable test functions. However, we face a new difficulty in the global existence of the associated flow, which we describe hereafter.

We note that \eqref{Toda} is the Euler-Lagrangian equation for the following functional
\begin{equation*}
\tilde{I}_{\rho_1,\rho_2}(u_1(x,t),u_2(x,t))=\int_{\Sigma}Q(u_1,u_2)dV_g-\rho_1 \log\int_{\Sigma}e^{u_1-\overline{u}_1}dV_g-\rho_2 \log\int_{\Sigma}e^{u_2-\overline{u}_2}dV_g,
\end{equation*}
where $Q(u_1,u_2)={\frac{1}{3}}\big(|\nabla_gu_1|^2+|\nabla_gu_2|^2+\nabla_gu_1\nabla_gu_2\big)$.

One interesting point is that a possible gradient flow of $\tilde{I}_{\rho_1,\rho_2}$ is the following one, quite different from standard (semilinear) parabolic equations,
\begin{equation}\label{E-Toda}
\left\{\begin{aligned}
\frac{\partial u_1}{\partial t}=\frac23 e^{-u_1}\Delta_gu_1 +\frac13 e^{-u_1}\Delta_gu_2+4\pi\left(\frac{h_1}{\int_{\Sigma}h_1e^{u_1}dV_g}-e^{-u_1}\right),\\
\frac{\partial u_2}{\partial t}=\frac13 e^{-u_2}\Delta_gu_1+\frac23 e^{-u_2}\Delta_gu_2+\rho_2\left(\frac{h_2}{\int_{\Sigma}h_2e^{u_2}dV_g}-e^{-u_2}\right).
\end{aligned}\right.
\end{equation}
To study the global existence we exploit the following idea. Observe that the eigenvalues of the matrix
\begin{equation*}
\begin{pmatrix}
\frac23 e^{-w_1(x,t)} & \frac13 e^{-w_2(x,t)}\\
\frac13 e^{-w_1(x,t)} & \frac23 e^{-w_2(x,t)}
\end{pmatrix},
\quad w_i(x,t) \in C^{2+\alpha,1+\alpha/2}(\Sigma \times [0,\infty)) : \text{ fixed functions }
\end{equation*}
are positive and distinct. Thus, one can use the eigenvalues $\lambda_i(x,t)$ and their eigenvectors to transform the linearized operator of \eqref{E-Toda} at $(w_1, w_2)$ into the standard form alike
\[\frac{\partial \phi_i}{\partial t}-\lambda_i(x,t) \Delta \phi_i + L_i(x,t,\phi_1,\phi_2,\nabla \phi_1,\nabla \phi_2), \quad \lambda_i(x,t)\ge c>0 \text{ for } x\in \Sigma, t\in [0,T], \ \ i=1,2.\]
From this observation, we can apply the standard parabolic theory to prove the short time existence of \eqref{E-Toda}.
(We refer to \cite[Chap. 9]{Friedman} for the linear parabolic systems and \cite[Chap. 7]{Friedman}, \cite{Huisken} for the quasilinear parabolic equations).

However, there is an obstacle when proving a priori estimates for the original quasilinear system.
In order to apply the well-known estimates for the standard parabolic equations (e.g. Schauder estimate), we need to transform the system again.
Observe that the coefficient matrix involves $e^{-u_1}$, $e^{-u_2}$, and the coefficients of the transformed system are related to $e^{-u_i}$ and their derivatives.
Since the constant in the Schauder estimate depends on the norm of the coefficients of the differential operator, we could not apply the standard estimate.
Due to this difficulty, at this point we can not prove the global existence of the flow \eqref{E-Toda}. We postpone this to a future work.

We remark that, in \cite{ParkZhang}, the authors proved the global existence of the solution of another semilinear parabolic system related to elliptic systems including Toda systems and Liouville systems.
\end{remark}

The organization of this paper is as follows. In Section 2, we prove the the global existence of solutions to \eqref{maineq}.
In Section 3, we will carry out blow-up analysis and derive the lower bound of blow-up sequences, which is the key element in proving the existence result of \eqref{maineq-s}.
In Section 4, we construct a test function and we prove the convergence of the flow. This completes the proof of Theorem \ref{main-thm}.

\section{Global existence of the flow}

In this section, we prove the global existence and uniqueness of solutions to the flow \eqref{maineq}.
The argument is divided into two main steps: we first recall the short-time existence and present several preliminary properties of the solution, and then derive a priori estimates that allow us to extend the solution globally in time.

\subsection{Short-time existence and preliminaries}

As the first step, by the standard parabolic theory, we can prove the short-time existence and uniqueness of the solution to \eqref{maineq} (for example, see \cite{Friedman}). We omit the details and refer the reader to the references.

\begin{lemma}\label{short-time}
Fix $\alpha \in (0,1)$. For any initial data $u_0 \in C^{2+\alpha}(\Sigma)$, there exists $\varepsilon>0$ such that \eqref{maineq} has a unique solution $u \in C^{2+\alpha,1+\alpha/2}(\Sigma \times [0,\varepsilon])$.
\end{lemma}

We first note that by Lemma \ref{short-time}, there exists $T>0$ such that \eqref{maineq} has a unique solution $u\in C^{2+\alpha,1+\alpha/2}(\Sigma \times [0,T])$.
In addition, we prove several basic properties of the solution, including conservation of mass and monotonicity of the energy functional. 
These preliminaries are crucial for the energy method employed in the next subsection and will also be used later in the paper.

\begin{lemma}\label{V-conserva} Suppose that \eqref{maineq} admits a solution $u \in C^{2+\alpha,1+\alpha/2}(\Sigma \times [0,T])$ for some $T>0$. 
Then the following properties hold:
{\rm{(i)}} For all $t\in[0,\ T],$ we have
\[\int_{\Sigma}e^{u(t)}dV_g=\int_{\Sigma}e^{u_0}dV_g.\]
{\rm{(ii)}} The energy functional $J_{\rho_2}(u(t))$ is non-increasing in $t$, that is, 
for all $0\leq t_0\leq t_1\leq T,$
\[
J_{\rho_2}(u(t_1))\leq J_{\rho_2}(u(t_0)).
\]
\end{lemma}
\begin{proof}
\noindent \textbf{(i)} By integrating both sides of \eqref{maineq} over $\Sigma \times [0,t]$, we have
\begin{equation*}
	0=\int_{\Sigma}\int^t_0{\frac{\partial}{\partial t}}(e^{u(t)})dsdV_g=\int_{\Sigma}e^{u(t)}dV_g-\int_{\Sigma}e^{u(0)}dV_g=\int_{\Sigma}e^{u(t)}dV_g-\int_{\Sigma}e^{u_0}dV_g.
\end{equation*}
		
\noindent \textbf{(ii)}
By differentiating $J_{\rho_2}(u(t))$ with respect to $t$ and integrating by parts, we get
\begin{equation}\label{J'}
	{\frac{\partial}{\partial t}}J_{\rho_2}(u(t))=-\int_{\Sigma}\left|\frac{\partial u}{\partial t}\right|^2 e^u dV_g \leq 0.
\end{equation}
\par

Integrating \eqref{J'} with respect to time from\ $t_0$ to $t_1$, we obtain that 
\begin{equation*}
J_{\rho_2}(u(t_1))-J_{\rho_2}(u(t_0))=-\int^{t_1}_{t_0}\int_{\Sigma} \left|\frac{\partial u}{\partial t}\right|^2e^udV_gdt\leq 0.
\end{equation*}
\end{proof}

\begin{lemma}\label{Bounded-F}
There exist $C,c>0$, independent of $T$, such that for all $t \in [0,T]$,
\[\int_{\Sigma}h_1e^udV_g,\ \frac{\int_{\Sigma}e^{-u}dV_g}{\int_{\Sigma}h_2e^{-u}dV_g} \le C \quad \text{ and } \quad \int_\Sigma h_1 e^u dV_g, \int_{\Sigma}h_2e^{-u}dV_g \ge c.\]
\end{lemma}
\begin{proof}
First, we prove that $ c \le \int_{\Sigma}h_1e^udV_g \le C$ for some constants $C,c>0$, independent of $t \in [0,T]$. 
Indeed, by Moser-Trudinger inequality \eqref{MTineq} and Jensen's inequality, it holds that
\begin{equation}\label{MTineq2}
	\frac12\int_{\Sigma}|\nabla_gu|^2dV_g \ge 8\pi\log\int_{\Sigma}e^{u-\overline{u}}dV_g+\rho_2\log\int_{\Sigma}e^{-u+\overline{u}}dV_g  - C_{\Sigma}.
\end{equation}
Employing \eqref{MTineq2} and the monotonicity of $J_{\rho_2}(u(t))$, we obtain that for all $t \in [0,T]$,
\begin{equation}\label{bf3.1}
    \begin{aligned}
		J_{\rho_2}(u_0)\geq J_{\rho_2}(u(t))
		\geq& \,8\pi\log\int_{\Sigma}e^{u}dV_g+ \rho_2 \log\int_{\Sigma}e^{-u}dV_g-8\pi\log\int_{\Sigma}h_1e^{u}dV_g\\
        &\quad-\rho_2\log\int_{\Sigma}h_2e^{-u}dV_g-C_{\Sigma}.
	\end{aligned}
\end{equation}
This implies that $J_{\rho_2}(u(t)) \geq8\pi\log\int_{\Sigma}e^{u}dV_g -8\pi\log\int_{\Sigma}h_1e^udV_g -\rho_2\log\|h_2\|_{L^\infty(\Sigma)} -C_{\Sigma}$. 
From this inequality, we derive that
\begin{equation}\label{bf3.2}
	0<\exp\Big(-\frac{\rho_2\log\|h_2\|_{L^\infty(\Sigma)} +J_{\rho_2}(u_0)+C_{\Sigma}}{8\pi}\Big)\int_{\Sigma}e^{u_0}dV_g\leq\int_{\Sigma}h_1e^{u}dV_g\leq\|h_1\|_{L^{\infty}(\Sigma)}\int_{\Sigma}e^{u_0}dV_g.
\end{equation}

Next, we prove the uniform boundedness of $\frac{\int_{\Sigma}e^{-u}dV_g}{\int_{\Sigma}h_2e^{-u}dV_g}$. 
From \eqref{bf3.1}, we also deduce that $J_{\rho_2}(u(t)) \geq \rho_2\log\int_{\Sigma}e^{-u}dV_g-\rho_2\log\int_{\Sigma}h_2e^{-u}dV_g-8\pi\log\|h_1\|_{L^\infty(\Sigma)}-C_{\Sigma}$. 
This implies that for all $t\in [0,T]$
\begin{equation*}\label{bf3.4}
	0<{\frac{\int_{\Sigma}e^{-u}dV_g}{\int_{\Sigma}h_2e^{-u}dV_g}}\leq \exp\Big(\frac{J_{\rho_2}(u_0)+8\pi\log\|h_1\|_{L^\infty(\Sigma)}+C_{\Sigma}}{\rho_2}\Big).
\end{equation*}

Finally, we show that $\int_{\Sigma}h_2e^{-u}dV_g \ge c$ for some $c>0$, independent of $t \in [0,T]$. Indeed, by the Cauchy-Schwarz inequality and \eqref{bf3.2}, we obtain
\begin{equation*}\label{bf3.6}
	\int_{\Sigma}h_2e^{-u}dV_g \geq \frac{\big(\int_{\Sigma}\sqrt{h_1h_2}dV_g\big)^2}{\int_{\Sigma} h_1 e^u dV_g} \ge {\frac{\big(\int_{\Sigma}\sqrt{h_1h_2}dV_g\big)^2}{\|h_1\|_{L^{\infty}(\Sigma)}\cdot\int_{\Sigma}e^{u_0}dV_g}} >0.
\end{equation*}
This completes the proof of Lemma \ref{Bounded-F}.
\end{proof}

\medskip

\subsection{A priori estimates and global existence}

We now address the global existence of solutions to the flow \eqref{maineq}, i.e. Theorem \ref{global}.

Let us define the maximal time $T_0>0$ by
\begin{equation*}
T_0:=\sup\Big\{T>0: u\in C^{2+\alpha,1+\alpha/2}(\Sigma \times [0,T]) 
\text{ is the unique solution of \eqref{maineq}} \Big\}.
\end{equation*}
To prove global existence, it suffices to show that $T_0=\infty$. 
For this purpose, we first derive a priori estimates for the solution of \eqref{maineq} in three steps: Step 1. $H^1$-estimate (Proposition \ref{H1}); Step 2. $H^2$-estimate (Proposition \ref{H2}); Step 3. $C^{2+\alpha,1+\alpha/2}$-estimate (Proposition \ref{C2alpha}).

These three steps provide progressively stronger control over the solution, which ultimately allows us to extend the solution beyond any finite maximal time.

\begin{prop}\label{H1}
Let $u$ be the solution of \eqref{maineq} on $[0,T)$ for some $T>0$.
Then there exists a constant $C_{T,1}=C(T,\|u_0\|_{H^1(\Sigma)})$, depending on $\Sigma$, $h_1$ and $h_2$, such that
\[\|u(t)\|_{H^1(\Sigma)} \le C_{T,1}, \quad \forall t \in [0,T).\]
\end{prop}
\begin{proof}
Since $J_{\rho_2}(u(t))$ is decreasing in $t$, by subtracting a multiple of the Moser-Trudinger inequality \eqref{MTineq} from the definition of $J_{\rho_2}(u(t))$, we obtain
\begin{equation}\label{H1.3}
\begin{aligned}
J_{\rho_2}(u_0) \ge J_{\rho_2}(u(t))&={\frac{1}{2}}\int_{\Sigma}|\nabla_gu|^2dV_g-8\pi\log\int_{\Sigma}h_1e^{u}dV_g-\rho_2\log\int_{\Sigma}h_2e^{-u}dV_g+(8\pi-\rho_2)\overline{u}\\
&\geq{\frac{8\pi-\rho_2}{16\pi}}\int_{\Sigma}|\nabla_gu|^2dV_g-\big(8\pi-\rho_2\big)\log\int_{\Sigma}e^udV_g+\big(8\pi-\rho_2\big)\overline{u}-C
\end{aligned}
\end{equation}
for a constant $C$ independent of $u_0$ and $t$.
Since $\int_\Sigma e^{u(t)}dV_g=\int_\Sigma e^{u_0}dV_g$, applying Young's inequality, we obtain that for small $\epsilon>0$
\begin{equation}\label{H1.4}
\begin{aligned}
\|\nabla_gu(t)\|^2_{L^2(\Sigma)} & \le \frac{16\pi}{8\pi-\rho_2}J_{\rho_2}(u_0)-16\pi\overline{u}(t)+16\pi\log\int_{\Sigma}e^{u_0}dV_g+C\\
&\leq C(\|u_0\|_{H^1(\Sigma)}) + \epsilon\|u\|^2_{L^2(\Sigma)}+C_{\epsilon}
\end{aligned}
\end{equation}
where $C(\|u_0\|_{H_1(\Sigma)})$ denotes a constant depending only on $\|u_0\|_{H^1(\Sigma)}$, and $C_\epsilon$ is a constant depending on $\epsilon$.

Next, differentiating $\int_{\Sigma}e^{2u}dV_g$ with respect to $t$, we obtain that there exists a constant $C=C(\|u_0\|_{H^1(\Sigma)})$ such that
\begin{equation*}\label{H1.1}
\begin{aligned}
\frac12{\frac{d}{dt}}\int_{\Sigma}e^{2u(t)}dV_g&=\int_{\Sigma}e^{u}\Big[\Delta_gu(t)+8\pi\big({\frac{h_1e^u}{\int_{\Sigma}h_1e^udV_g}}-1\big)-\rho_2
\big(\frac{h_2e^{-u}}{\int_{\Sigma}h_2e^{-u}dV_g}-1\big)\Big]dV_g\\
&=-\int_{\Sigma}e^{u}|\nabla_gu|^2dV_g+8\pi{\frac{\int_{\Sigma}h_1e^{2u}dV_g}{\int_{\Sigma}h_1e^udV_g}}-\rho_2{\frac{\int_{\Sigma}h_2dV_g}{\int_{\Sigma}h_2e^{-u}dV_g}}-(8\pi-\rho_2)\int_{\Sigma}e^{u}dV_g\\
&\leq C\int_{\Sigma}e^{2u(t)}dV_g + C
\end{aligned}
\end{equation*}
Integrating this differential inequality, we conclude that
\begin{equation}\label{H1.2}
\int_{\Sigma}e^{2u(t)}dV_g\leq e^{Ct}\int_{\Sigma}e^{2u_0}dV_g + e^{Ct}\leq C(T,\|u_0\|_{H^1(\Sigma)}) \ \ \text{ for } t \in [0,T).
\end{equation}
In order to estimate the average value of $u(t)$, set
\begin{equation*}
A(t):=\Big\{x\in \Sigma : e^{u(x,t)} \ge \frac12 \int_\Sigma e^{u(t)}dV_g\Big\}.
\end{equation*}

By H\"older's inequality and \eqref{H1.2}, it follows that
\begin{equation*}
\int_\Sigma e^{u_0}dV_g = \int_{\Sigma\setminus A(t)} e^{u(t)}dV_g + \int_{A(t)} e^{u(t)}dV_g \le \frac12 \int_\Sigma e^{u_0} dV_g + C(T,\|u_0\|_{H^1}) |A(t)|_g^{1/2}.
\end{equation*}
Therefore, there exists $c_T>0$ such that
\[c_T \le |A(t)| \le 1, \ \ \text{ and } \ \  |A(t)|\log\frac{\int_\Sigma e^{u_0}dV_g}{2} \le \int_{A(t)} u(t) dV_g \le \int_{A(t)} e^{u(t)}dV_g.\]
Consequently, by H\"older's inequality and \eqref{H1.2}, we have
\begin{align*}
|\overline{u}(t)| &\le |\Sigma\setminus A(t)|^{1/2} \|u(t)\|_{L^2(\Sigma\setminus A(t))} + \Big|\int_{A(t)} u(t)dV_g\Big|\le \sqrt{1-c_T} \|u(t)\|_{L^2(\Sigma)}+C(T,\|u_0\|_{H^1}).
\end{align*}
Combining this estimate with the Poincar\'e inequality, we deduce
\[ \|u(t)\|_{L^2} \le  c\|\nabla u(t)\|_{L^2}+ \sqrt{1-c_T} \|u(t)\|_{L^2(\Sigma)}+C(T,\|u_0\|_{H^1}),\]
which implies
\begin{equation}\label{H1.6}
\|u(t)\|_{L^2} \le C \|\nabla u(t)\|_{L^2}+ C(T,\|u_0\|_{H^1}).
\end{equation}
Finally, combining \eqref{H1.4} and \eqref{H1.6}, and choosing $\epsilon$ sufficiently small, we complete the proof.
\end{proof}

\begin{prop}\label{H2}
Let $u$ be the solution of \eqref{maineq} on $[0,T)$ for some $T>0$.
Then there exists a constant $C_{T,2}=C(T,\|u_0\|_{H^2(\Sigma)})$, depending on $\Sigma$, $h_1$ and $h_2$, such that
\[\|u(t)\|_{H^2(\Sigma)} \le C_{T,2}, \quad \forall t \in [0,T).\]
\end{prop}
\begin{proof}
By Proposition \ref{H1}, it suffices to estimate $\|\Delta_g u(t)\|_{L^2(\Sigma)}$.
To this end, we introduce the auxiliary function $\nu(t)=\frac{\partial u(t)}{\partial t}e^{u(t)/2}$.
Then, by a direct computation we obtain
\begin{equation}\label{H2.5}
\begin{aligned}
&\frac12\frac{d}{dt}\int_{\Sigma}\Big(1+\big|\Delta_gu(t)\big|^2\Big)dV_g= \int_{\Sigma}\Delta_gu(t)\Delta_g\big({\frac{\partial u(t)}{\partial t}}\big)dV_g\\
&=\int_{\Sigma}\Big(e^{\frac{u}{2}}\nu(t)-8\pi (\frac{h_1e^u}{\int_{\Sigma}h_1e^udV_g}-1)+\rho_2\big(\frac{h_2e^{-u}}{\int_{\Sigma}h_2e^{-u}dV_g}-1\big)\Big)\Delta_g\big(e^{-\frac{u}{2}}\nu(t)\big)dV_g\\
&=-\int_{\Sigma}|\nabla_g\nu(t)|^2dV_g+{\frac{1}{4}}\int_{\Sigma}{\nu(t)}^2|\nabla_gu(t)|^2dV_g\\
&\quad +{\frac{8\pi}{\int_{\Sigma}h_1e^{u}dV_g}}\int_{\Sigma}e^{\frac{u(t)}{2}}\big(\nabla_gh_1+h_1\nabla_gu\big)\big(\nabla_g\nu(t) -\frac12\nu(t)\nabla_g u\big)dV_g\\
&\quad-{\frac{\rho_2}{\int_{\Sigma}h_2e^{-u}dV_g}}\int_{\Sigma}e^{-\frac{3u}{2}}(\nabla_gh_2-h_2\nabla_gu)\big(\nabla_g\nu(t)-{\frac{1}{2}}\nu(t)\nabla_gu\big)dV_g.
\end{aligned}
\end{equation}

By Lemma \ref{Bounded-F}, Proposition \ref{H1} and the Moser-Trudinger inequality, there exists a constant $C_{T},$ such that for all $p\geq1,\ t\in[0,T)$
\begin{equation}\label{H2.1}
\frac{1}{\int_\Sigma h_1 e^{u(t)}dV_g}+\frac{1}{\int_\Sigma h_2 e^{-u(t)}dV_g}+\int_\Sigma e^{pu(t)}dV_g+\int_\Sigma e^{-pu(t)}dV_g \le C_T.
\end{equation}
Now, using \eqref{H2.1} together with Hölder’s inequality and Young’s inequality, we can estimate the third term on the right-hand side of \eqref{H2.5} as follows:
\begin{equation}\label{H2.6}
\begin{aligned}
&{\frac{8\pi}{\int_{\Sigma}h_1e^{u}dV_g}}\int_{\Sigma}e^{\frac{u(t)}{2}}\big(\nabla_gh_1+h_1\nabla_gu(t)\big)\big(\nabla_g\nu(t)-{\frac{1}{2}}\nu(t)\nabla_gu(t)\big)dV_g\\
\leq& C{\frac{\big(\int_{\Sigma}e^{2u}dV_g\big)^{\frac{1}{4}}}{\int_{\Sigma}h_1e^{u}dV_g}}\Big(\int_{\Sigma}\big|\nabla_gh_1+h_1\nabla_gu(t)\big|^4dV_g\Big)^{\frac{1}{4}}\Big(\int_{\Sigma}|\nabla_g\nu(t)-{\frac{1}{2}}\nu(t)\nabla_gu(t)|^2dV_g\Big)^{\frac{1}{2}}\\
\leq& C_{T}\big(1+\|\nabla_gu\|_{L^4(\Sigma)}\big)\left\{\|\nabla_g\nu\|_{L^2(\Sigma)}+\Big(\int_{\Sigma}{\nu(t)}^2|\nabla_gu|^2dV_g\Big)^{\frac{1}{2}}\right\}\\
\leq& \epsilon\|\nabla_g\nu\|^2_{L^2(\Sigma)}+\epsilon\int_{\Sigma}{\nu(t)}^2|\nabla_gu(t)|^2dV_g+C_{\epsilon,T}\big(1+\|\nabla_gu\|^2_{L^4(\Sigma)}\big),
\end{aligned}
\end{equation}
where $C_{\epsilon,T}$ depends only on $\Sigma,T,\epsilon,\|u_0\|_{H^1(\Sigma)}.$
Similarly, we can estimate the fourth term on right-hand side of \eqref{H2.5}. More precisely,
\begin{equation}\label{H2.7}
\begin{aligned}
&-{\frac{\rho_2}{\int_{\Sigma}h_2e^{-u}dV_g}}\int_{\Sigma}e^{-\frac{3u(t)}{2}}(\nabla_gh_2-h_2\nabla_gu(t))\big(\nabla_g\nu(t)-{\frac{1}{2}}\nu(t)\nabla_gu(t)\big)dV_g\\
&\leq \epsilon\|\nabla_g\nu\|^2_{L^2(\Sigma)}+\epsilon\int_{\Sigma}{\nu(t)}^2|\nabla_gu(t)|^2dV_g+C_{\epsilon,T}\big(1+\|\nabla_gu\|^2_{L^4(\Sigma)}\big).
\end{aligned}
\end{equation}
\par

Combining \eqref{H2.6} and \eqref{H2.7}, and choosing $\epsilon>0$ sufficiently small, we simplify \eqref{H2.5} into
\begin{equation}\label{H2.7*}
\begin{aligned}
{\frac{1}{2}}{\frac{d}{dt}}\int_{\Sigma}\Big(1+\big|\Delta_gu(t)\big|^2\Big)dV_g \leq- \frac12\|\nabla_g\nu(t)\|^2_{L^2(\Sigma)}+\int_{\Sigma}{\nu(t)}^2|\nabla_gu|^2dV_g+C_{\epsilon,T}\big(1+\|\nabla_gu\|^2_{L^4(\Sigma)}\big).
\end{aligned}
\end{equation}
On the other hand, by Proposition \ref{H1}, H\"older's inequality and Gagliardo-Nirenberg inequality, we deduce 
\begin{equation}\label{H2.8}
\begin{aligned}
\int_{\Sigma}\nu^2(t)|\nabla_gu|^2dV_g&\leq C\|\nu(t)\|_{L^2(\Sigma)}\|\nu(t)\|_{H^1(\Sigma)}\|u(t)\|_{H^1(\Sigma)}\|u(t)\|_{H^2(\Sigma)}\\
&\leq C\big(T,\|u_0\|_{H^1(\Sigma)}\big)\|\nu(t)\|_{L^2(\Sigma)}\|\nu(t)\|_{H^1(\Sigma)}\|u(t)\|_{H^2(\Sigma)}\\
&\leq\epsilon\big(\|\nabla_g\nu\|^2_{L^2(\Sigma)}+\|\nu\|^2_{L^2(\Sigma)}\big)+C_{\epsilon,T}\|\nu\|^2_{L^2(\Sigma)}\big(\|\Delta_gu(t)\|^2_{L^2(\Sigma)}+1\big),\\
\end{aligned}
\end{equation}
and similarly
\begin{equation}\label{H2.9}
\begin{aligned}
\|\nabla_gu\|^2_{L^4(\Sigma)}&\leq C\|u\|_{H^1(\Sigma)}\|u\|_{H^2(\Sigma)} \leq C\big(T,\|u_0\|_{H^1(\Sigma)}\big)\big(\|\Delta_gu(t)\|^2_{L^2(\Sigma)}+1\big).
\end{aligned}
\end{equation}
\par

Therefore, combining \eqref{H2.7*}, \eqref{H2.8} and \eqref{H2.9} together, we obtain
\begin{equation*}
\begin{aligned}
{\frac{d}{dt}}\int_{\Sigma}\Big(1+|\Delta_gu|^2\Big)dV_g 
&\leq C\big(T,\|u_0\|_{H^1(\Sigma)}\big)\big(1+\|\nu\|^2_{L^2(\Sigma)}\big)\big(1+\|\Delta_gu\|^2_{L^2(\Sigma)}\big).
\end{aligned}
\end{equation*}
As a consequence, using the energy identity $J_{\rho_2}(u(T))-J_{\rho_2}(u_0)=-\int_0^T\int_{\Sigma}|{\frac{\partial u}{\partial t}}|^2e^{u}dV_gdt$ and integrating in time, we obtain
\begin{equation*}
\begin{aligned}
\log\Big(1+\|\Delta_gu\|^2_{L^2(\Sigma)}\Big)&\leq C\big(T,\|u_0\|_{H^1(\Sigma)}\big)\Big(1+\int_0^T\big(1+\|{\frac{\partial u}{\partial t}}e^{\frac{u}{2}}\|^2_{L^2(\Sigma)}\big)dt\Big)\\
&\leq C\big(T,\|u_0\|_{H^1(\Sigma)}\big)\Big(1+J_{\rho_2}(u_0)-J_{\rho_2}(u(T))\Big)\leq C\big(T,\|u_0\|_{H^1(\Sigma)}\big).
\end{aligned}
\end{equation*}
This completes the proof.
\end{proof}

\begin{prop}\label{C2alpha}
Let $u$ be the solution of \eqref{maineq} on $[0,T)$ for some $T>0$.
Then there exists a constant $C_{T,3}=C(T,\|u_0\|_{C^{2+\alpha}(\Sigma)})$, depending on $\Sigma$, $h_1$, $h_2$ such that
\[\|u(t)\|_{C^{2+\alpha,1+\alpha/2}(\Sigma \times [0,T))} \le C_{T,3}.\]
\end{prop}
\begin{proof}
For the proof of this proposition, it suffices to show the following estimate:
\begin{equation}\label{c3.0}
|u(x,t_1)-u(y,t_2)|\leq C\Big(|x-y|^{\alpha}+|t_1-t_2|^{\frac{\alpha}{2}}\Big),\ \ \text{for any}\ x,y\in\Sigma;\ \ t_1,t_2\in[0,T),
\end{equation}
where $|x-y|$ denotes $\dist_g(x,y)$ for simplicity.
In view of \eqref{c3.0}, the classical parabolic Schauder estimates (see, e.g., \cite[Chapter 3]{Friedman}) yield the desired conclusion.

Now we prove \eqref{c3.0}.
By Proposition \ref{H2}, we have $\|u(t)\|_{H^2(\Sigma)} \le C_{T,2}$, and by Sobolev embedding theorem, there exists a constant $C_1$ such that $\|u(t)\|_{C^{0,\alpha}(\Sigma)} \le C_1$ for all $t\in [0,T)$.
Therefore, it is enough to prove
\begin{equation}\label{c3.1}
\big|u(x,t_1)-u(x,t_2)\big| \le C |t_1-t_2|^{\frac{\alpha}{2}} \ \ \text{ for all } \ \ x \in \Sigma, t_1, t_2 \in [0,T).
\end{equation}
(i) If $t_2-t_1\geq1,$ then we have
\begin{equation}\label{c3.2}
\big|u(x,t_1)-u(x,t_2)\big| \leq C(T,\|u_0\|_{H^2(\Sigma)}) \le C(T,\|u_0\|_{H^2(\Sigma)})\big|t_1-t_2\big|^{\frac{\alpha}{2}}.
\end{equation}
(ii) If $0<t_2-t_1<1,$ set $s=\min\{r_0/2,\sqrt{t_2-t_1}\}$, where $r_0$ is the injectivity radius of $\Sigma$. Then, we have 
\begin{equation}\label{c3.3}
\begin{aligned}
\big|u(x,t_1)-u(x,t_2)\big|&={\frac{1}{|B_{s}(x)|}}\int_{B_{s}(x)}\big|u(x,t_1)-u(x,t_2)\big|dV_g(y)\\
&\leq C \int_{B_{s}(x)}\sum_{i=1}^2 \frac{\big|u(x,t_i)-u(y,t_i)\big|}{t_2-t_1}+\frac{\big|u(y,t_1)-u(y,t_2)\big|}{t_2-t_1}dV_g(y),
\end{aligned}
\end{equation}
where we used $cs^2 \le |B_s(x)| \le C s^2$ on $(\Sigma,g)$ with constants $C,c>0$ depending only on $(\Sigma,g)$.

Let us calculate the first term of the right-hand side on \eqref{c3.3}. By H\"older continuity $\|u(t)\|_{C^{0,\alpha}(\Sigma)} \le C_1$, we obtain that, for $i=1,2$,
\begin{equation}\label{c3.4}
\begin{aligned}
&\int_{B_{s}(x)} \frac{\big|u(x,t_i)-u(y,t_i)\big|}{t_2-t_1}dV_g(y)  \leq C\int_{B_{s}(x)}\frac{\big|x-y\big|^{\alpha}}{t_2-t_1}dV_g(y)\leq C\frac{s^{\alpha+2}}{t_2-t_1} \leq C\big(t_2-t_1\big)^{\frac{\alpha}{2}}.
\end{aligned}
\end{equation}
For the second term, we obtain
\begin{equation}\label{c3.7}
\begin{aligned}
\int_{B_{s}(x)}\frac{\big|u(y,t_1)-u(y,t_2)\big|}{t_2-t_1}dV_g(y)
&\leq C\sup\limits_{t_1\leq t\leq t_2}\int_{B_{s}(x)}\Big|{\frac{\partial u(t)}{\partial t}}\Big|dV_g(y)\\
&\leq C|B_s(x)|^{1/2} \sup\limits_{t_1\leq t\leq t_2}\big(\int_{B_{s}(x)}\Big|\frac{\partial u(t)}{\partial t}\Big|^2dV_g(y)\big)^{\frac{1}{2}}\\
&\leq C(T,\|u_0\|_{H^2})\big(t_2-t_1\big)^{\frac{\alpha}{2}}.
\end{aligned}
\end{equation}
since $\|\Delta_g u(t)\|_{L^2}, \|u(t)\|_{L^\infty} \le C(T,\|u_0\|_{H^2})$ and $\int_\Sigma h_1 e^u dV_g, \int_\Sigma h_2 e^{-u}dV_g \ge c>0$ by Lemma \ref{Bounded-F} and Proposition \ref{H2}.
Combining \eqref{c3.2}--\eqref{c3.7}, we obtain \eqref{c3.1} and this completes the proof.
\end{proof}

We now complete the proof of the global existence.

\begin{proof}[Proof of Theorem \ref{global}.]
Suppose, by contradiction, that $T_0 < \infty$. 
By the a priori estimates in Proposition \ref{C2alpha}, the solution $u(t)$ remains bounded in $C^{2+\alpha}(\Sigma)$ up to $t=T_0$, 
and hence the short-time existence lemma guarantees that $u$ can be extended beyond $T_0$. 
This contradicts the definition of $T_0$ as the maximal existence time.
Therefore, we conclude that $T_0 = \infty$.
This completes the proof of the global existence and uniqueness of solutions to the flow \eqref{maineq}.
\end{proof}

\section{Blow-up Analysis}
In this section we investigate the blow-up behavior of the flow \eqref{maineq}.
We choose a sequence of times $t_n\to\infty$ (see \eqref{Vc}) and study the behavior of $u(t_n)$.
We determine the number of blow-up points on $\Sigma$ and establish a uniform upper bound for the second (normalized) component (Proposition~\ref{Never-S}).
We also derive an energy lower bound in the blow-up regime.
All normalizations and rescalings used for blow-up subsequences will be introduced where they are first needed.

To this end, we extract a sequence $t_n\to\infty$ along which the time–derivative term vanishes in a suitable sense. 
Since $J_{\rho_2}(u(t))$ is nonincreasing in $t$ (Lemma~\ref{V-conserva}(ii)) and bounded from below, we have
\begin{equation}\label{sec3eq1}
J_{\rho_2}(u(0))-\lim_{t\to\infty} J_{\rho_2}(u(t))
= \int_0^\infty\!\!\int_{\Sigma}\Bigl|\frac{\partial u}{\partial t}(t)\Bigr|^2 e^{u(t)}\,dV_g\,dt \;\le C.
\end{equation}
Hence there exists $t_n\to\infty$ such that
\begin{equation}\label{Vc}
\int_{\Sigma}\Bigl|\frac{\partial u}{\partial t}(t_n)\Bigr|^2 e^{u(t_n)}\,dV_g \;\longrightarrow\; 0 \quad\text{as } n\to\infty.
\end{equation}

For simplicity, set $u_n:=u(t_n)$ and introduce the normalized functions
\begin{equation}\label{Def-u1u2}
u_1^n := u_n - \log\!\int_{\Sigma} h_1 e^{u_n}\,dV_g,\qquad 
u_2^n := -\,u_n - \log\!\int_{\Sigma} h_2 e^{-u_n}\,dV_g,\qquad
f_n := \frac{\partial u}{\partial t}(t_n)\,e^{u_n/2}.
\end{equation}
With these notations, $u_n$ solves
\begin{equation}\label{Original.Eq}
\begin{aligned}
-\Delta_g u_n
&= 8\pi\,h_1 e^{u_1^n} - \rho_2\,h_2 e^{u_2^n} - (8\pi-\rho_2)\;-\; f_n\,e^{u_n/2}
\qquad \text{on }\Sigma.
\end{aligned}
\end{equation}
Moreover, the normalized functions $u_i^n$ and $f_n$ satisfy the identities
\begin{equation}\label{eq4.2}
\int_{\Sigma} h_1 e^{u_1^n}\,dV_g=1,\qquad 
\int_{\Sigma} h_2 e^{u_2^n}\,dV_g=1,\qquad 
\|f_n\|_{L^2(\Sigma)}^2=\int_{\Sigma}\Bigl|\frac{\partial u}{\partial t}(t_n)\Bigr|^2 e^{u_n}\,dV_g \to 0.
\end{equation}
In particular, \eqref{eq4.2} shows that the time–derivative term is negligible as $n\to\infty$.

Passing to a subsequence if necessary, we may assume that
\[
8\pi\,h_1 e^{u_1^n}\rightharpoonup \mu_1,\qquad 
\rho_2\,h_2 e^{u_2^n}\rightharpoonup \mu_2 \quad\text{in the sense of measures on }\Sigma.
\]
Define the singular set
\begin{equation*}
S:=\{x\in \Sigma : \mu_1(\{x\})+\mu_2(\{x\})  \ge 4\pi\}.
\end{equation*}
Since $\mu_1(\Sigma)=8\pi$ and $\mu_2(\Sigma)=\rho_2$, the singular set $S$ is finite. 
In the class of stationary mean-field models (including the sinh-Gordon equation and the Toda systems), it is well-known that $S$ coincides with the set of blow-up points
\begin{equation*}
	S_1:=\left\{x\in\Sigma: \exists x_n\rightarrow x \text{ with } \max(u_1^n(x_n),u_2^n(x_n))  \rightarrow+\infty\right\}.
\end{equation*}

In our setting, the equation contains an additional time–derivative term. 
For the sake of completeness, we include a proof that the identification $S=S_1$ still holds in this case. 
We recall the Brezis–Merle estimate \cite{BrMe}, and also refer to \cite{DJLW} for the result on surfaces.

\begin{lemma}[Lemma 2.7 in \cite{DJLW}]\label{Brezis-Merle}
Let $\Omega\subset\Sigma$ be a smooth domain. Assume that $u$ is a solution to a Dirichlet problem
\begin{equation*}
-\Delta_gu=f \ \ \text{in}\ \Omega, \quad u=0\ \  \text{on}\ \partial\Omega,
\end{equation*}
where $f\in L^1(\Omega).$ For every $0<\delta<4\pi,$ there is a constant $C$ depending only on $\delta$ and $\Omega$ such that
\begin{equation*}
\int_{\Omega}\exp\Big(\frac{(4\pi-\delta)|u|}{\|f\|_{L^1(\Omega)}}\Big)dV_g \leq C.
\end{equation*}
\end{lemma}

As a consequence of Lemma \ref{Brezis-Merle}, we first prove that $u_n-\overline{u}_n$ is uniformly bounded on every compact subset of $\Sigma\setminus S$, and we then deduce that the singular set $S$ coincides with the blow-up set $S_1$.

\begin{lemma}\label{Bounded}
(1) For $x\notin S$, there exist a geodesic ball $B_R^g(x) \subset\Sigma\setminus S$ and a constant $C>0$ such that
\begin{equation}\label{bdd4.1}
\|u_n-\overline{u}_n\|_{L^{\infty}(B_R^g(x))}\leq C \quad \text{for all } n \in\N.
\end{equation}
(2) $S=S_1$.
\end{lemma}
\begin{proof}
(1) By Lemma \ref{V-conserva} (i) and \eqref{eq4.2}, we have $\|f_n e^{u_n/2}\|_{L^1(\Sigma)} \le \|e^{u_n/2}\|_{L^2(\Sigma)} \|f_n\|_{L^2(\Sigma)}=\|f_n\| _{L^2(\Sigma)}\cdot\big(\int_{\Sigma}e^{u_0}dV_g\big)^{\frac{1}{2}}\to 0$ as $n \to\infty$. 
Fix $x\notin S$ and choose $R>0$ so small that $\overline{B^g_{4R}(x)}\subset \Sigma\setminus S$.
Then, by the definition of $S$, there exists some $\delta\in(0,4\pi)$ such that for sufficiently large $n$
\begin{equation*}\label{bdd4.4}
	\int_{B^g_{4R}(x)}\Big(\big|f_n e^{\frac12 u_n}\big|+8\pi h_1 e^{u_1^n}+\rho_2 h_2 e^{u_2^n}+8\pi-\rho_2\Big) dV_g<4\pi-2\delta.
\end{equation*}
       Let $\zeta_n$  be the solution of a Dirichlet problem
\begin{equation*}\label{bdd4.2}
\Delta_g \zeta_n=f_n e^{\frac12 u_n}-8\pi h_1e^{u_1^n}+\rho_2 h_2e^{u_2^n}+8\pi-\rho_2 \ \  \text{ in } B_{4R}^g(x), \quad \zeta_n=0 \ \ \text{ on } \partial B_{4R}^g(x).
\end{equation*}
Applying Lemma \ref{Brezis-Merle} to $\zeta_n$, we obtain that
\begin{equation}\label{bdd-avu-2}
	\|\zeta_n\|_{L^p(B^g_{4R}(x))}\leq\|e^{|\zeta_n|}\|_{L^p(B^g_{4R}(x))} \leq C, \quad \text{ for } p=\frac{4\pi-\delta}{4\pi-2\delta}>1
\end{equation}
with $C$ independent of $n$.

Set $\eta_n=u_n-\overline{u}_n-\zeta_n.$ Then $\eta_n$ is a harmonic function in $B^g_{4R}(x)$, so we have
\begin{equation}\label{bdd4.6}
\begin{aligned}
\|\eta_n\|_{L^{\infty}(B^g_{2R}(x))}&\leq C\|\eta_n\|_{L^1(B^g_{4R}(x))} \leq C\big(\|u_n-\overline{u}_n\|_{L^1(\Sigma)}+\|\zeta_n\|_{L^1(B^g_{4R}(x))}\big) \le C
\end{aligned}
\end{equation}
From Lemma \ref{V-conserva} (i), applying Jensen's inequality, we obtain $\overline{u}_n\leq e^{\overline{u}_n}\leq\int_{\Sigma}e^{u_n}dV_g=\int_{\Sigma}e^{u_0}dV_g\leq C$.
Combining this with \eqref{bdd-avu-2} and \eqref{bdd4.6} yields $ \|e^{u_n}\|_{L^p(B^g_{2R}(x))} \le C\|e^{|\zeta_n|}\|_{L^p(B^g_{2R}(x))}\leq C$.

Setting $s={\frac{2p}{p+1}}>1$, by H\"older's inequality, we obtain
\begin{equation*}\label{bdd4.7}
\begin{aligned}
\int_{B^g_{2R}(x)}|f_n e^{\frac12 u_n}|^sdV_g&\leq\Big(\int_{B^g_{2R}(x)}|f_n|^2dV_g\Big)^{\frac{s}{2}}\Big(\int_{B^g_{2R}(x)}e^{pu_n}dV_g\Big)^{1-{\frac{s}{2}}}\\
&\leq C\Big(\int_{B^g_{2R}(x)}|f_n|^2dV_g\Big)^{\frac{s}{2}}\rightarrow0,\ \ \text{as}\ n\rightarrow+\infty.
\end{aligned}
\end{equation*}
Applying $L^p$-estimates (see \cite[Theorem 9.11]{G-T}) and the Sobolev embedding, we deduce that $\{\zeta_n\}_{n\in\N}$ is bounded in $W^{2,s}(B^g_{R}(x))$ and $L^\infty(B_R^g(x))$. Thus, $u_n-\overline{u}_n=\zeta_n+\eta_n$ is bounded in $L^\infty(B_R^g(x))$.

\medskip

(2) 
First, we prove that $S\subset S_1$. If $x_1\notin S_1$, then there exist $R_1>0$  and $C>0$ such that $B^g_{R_1}(x_1)\subset\Sigma\backslash S_1$ and $\max\limits_{x\in B^g_{R_1}(x_1)}\{e^{u^n_1},\,e^{u^n_2}\}\leq C$. 
For any $0<r<R_1$,
\begin{equation*}
	8\pi\int_{B^g_r(x_1)}h_1e^{u^n_1}dV_g+\rho_2\int_{B^g_r(x_1)}h_2e^{u^n_2}dV_g\leq Cr^2\to0,\quad \text{as}\ r\to0,	
\end{equation*}
which implies that $x_1\notin S$.

Next, it suffices to show that $S_1\subset S$. Suppose $x_0 \notin S$. 
By Jensen's inequality and Lemma \ref{Bounded-F}, we have
\begin{equation}\label{bdd-avu-4}			 
    \overline{u}^n_1\leq e^{\overline{u}^n_1}\leq\int_{\Sigma}e^{u^n_1}dV_g={\frac{\int_{\Sigma}e^{u_0}dV_g}{\int_{\Sigma}h_1e^{u_n}dV_g}}\leq C,\quad \overline{u}^n_2\leq e^{\overline{u}^n_2}\leq\int_{\Sigma}e^{u^n_2}dV_g={\frac{\int_{\Sigma}e^{-u_n}dV_g}{\int_{\Sigma}h_2e^{-u_n}dV_g}}\leq C.
\end{equation}
Then, by \eqref{bdd4.1} and \eqref{bdd-avu-4}, it follows that for any $x \in B_R^g(x_0) \subset \Sigma\setminus S$, 
\begin{equation*}
	u_i^n(x)\leq e^{u_i^n} \leq \exp\Big(\|u_i^n -\overline{u}_i^n\|_{L^{\infty}(B_R^g(x_0))}+\overline{u}_i^n\Big) \leq C,\quad i=1,2.
\end{equation*}
Thus, $x_0\notin S_1$. We conclude $S=S_1$.
\end{proof}

\subsection{Asymptotic behavior of a blow-up sequence $(u^n_1,\,u^n_2)$}

We now study the asymptotic behavior of a blow-up sequence $(u_1^n,u_2^n)$ arising from \eqref{Def-u1u2}.
For $i=1,2$, let $x_i^n\in\Sigma$ be a maximum point of $u_i^n$, and set
\begin{equation}\label{defxn}
c_i^n:=\max_{x\in\Sigma} u_i^n(x)=u_i^n(x_i^n), 
\qquad r_i^n:=e^{-c_i^n/2}.
\end{equation}

Our analysis proceeds in two steps. 
First (Proposition~\ref{Selection}), we show that blow-up is concentrated at a single point and obtain global pointwise control in terms of the distance to this point.
Second (Proposition~\ref{Never-S}), we prove a uniform upper bound for the second component $u_2^n$ on $\Sigma$.
We begin with the first step:

\begin{prop}\label{Selection}
Let $(u^n_1,u^n_2)$ be a blow-up sequence. Then, up to a subsequence, the following hold:
\begin{itemize}
\item[(1)] $r_2^n/r_1^n \to \infty$, $c_1^n \to \infty$ as $n\to\infty$ and $h_1(x_0)>0$ where $x_0=\lim\limits_{n\to\infty} x_1^n$;
\item[(2)] There exists $C_1$, independent of $n$, such that
\begin{equation}\label{selection1}
u^n_1(x)+2\log\dist_g(x,x^n_1)\leq C_1,\ u^n_2(x)+2\log\dist_g(x,x^n_1)\leq C_1,\ \forall x\in\Sigma.
\end{equation}
In particular, $x_0$ is the unique blow-up point, in other words, $S=\{x_0\}$.
\end{itemize}
\end{prop}
\begin{proof}
\textbf{Proof of (1)} We first prove that $r_2^n/r_1^n \to \infty$. Suppose, by contradiction, that $r_2^n \le C r_1^n$ for some $C>0$.
Taking a subsequence, we may assume that $x_2^n \to x_0 \in S$ as $n \to \infty$.
Choose an isothermal coordinate system near $x_0$, which satisfies $g=e^{\psi(x)}|dx|^2$ and $\psi(x_0)=0$.
Since $S$ is finite, we can fix $\tilde{r}>0$ so small that $x_0=0$ is the unique blow-up point in $B_{\tilde{r}}(0)$.
\par

Set, for $i=1,2$,
\begin{equation*}
w_i^n(x):=u_i^n(x_2^n+r_2^n x)+2\log r_2^n \qquad x\in B_{\tilde{r}/r_2^n}(0) \subset \R^2.
\end{equation*}
Then, $w_2^n$ satisfies, in $B_{\tilde{r}/r_2^n}(0) \subset \R^2$,
\begin{equation}\label{sp1}
\begin{aligned}
-\Delta w^n_2(x)=&\rho_2 (h_2 e^{\psi})(x^n_2+r^n_2x)e^{w^n_2}-8\pi (h_1e^{\psi})(x^n_2+r^n_2x)e^{w^n_1}+\big(8\pi-\rho_2\big) e^{\psi(x^n_2+r^n_2x)-c_2^n}\\
&+(f_n e^{\psi})(x^n_2+r^n_2x)e^{w_1^n(x)/2-c_2^n/2} \|h_1e^{u_n}\|_{L^1(\Sigma)}^{\frac{1}{2}}.
\end{aligned}
\end{equation}
where $\Delta$ denotes the Laplacian in the chosen coordinate.

By the definition of $w_i^n$ and $r_2^n$, we have $w_2^n(0)=0$. Moreover, by \eqref{eq4.2}, there exists $C'>0$ such that $e^{w_1^n}, e^{w_2^n} \le C'$ on each fixed ball $B_R(0)$, and
\[ \|f_n(x^n_2+r^n_2x)e^{\psi(x^n_2+r^n_2x)-c_2^n/2}\|_{L^2(B_R(0))} \le \|f_n\|_{L^2(\Sigma)}\to 0 \quad \text{ for any } R>0.\]
By Harnack's type inequalities (see \cite[Theorems 9.20, 9.22]{G-T}) and $L^p$-estimates, $\{w^n_2\}_{n\in\N}$ is bounded in $H^2_{\rm{loc}}(\mathbb{R}^2)$.
Hence, up to a subsequence,
\begin{equation*}
w^n_2 \rightharpoonup w_2 \ \ \ \text{weakly in}\ H^2_{\rm{loc}}(\mathbb{R}^2),\quad w^n_2 \rightarrow w_2\ \ \ \text{in}\ C^{\alpha}_{\rm{loc}}(\mathbb{R}^2).
\end{equation*}

On the other hand, by Lemma \ref{Bounded-F}, we have $\|h_1 e^{u_n}\|_{L^1(\Sigma)}, \|h_2 e^{u_n}\|_{L^1(\Sigma)} \ge c$ for all $n \in \mathbb{N}$, and thus
\begin{equation}\label{sp2}
w^n_1(x)+w^n_2(x) \le 4\log r^n_2 - \log(\|h_1 e^{u_n}\|_{L^1(\Sigma)}\|h_2 e^{-u_n}\|_{L^1(\Sigma)}) \rightarrow -\infty,
\end{equation}
uniformly on compact subsets of $\mathbb{R}^2$.
In particular, $w^n_1 \rightarrow-\infty$ locally uniformly.

Taking the limit in \eqref{sp1} on each $B_R(0)$ (using $e^{\psi(x_2^n+r_2^n x)}\to 1$, $(h_i e^{\psi})(x_2^n+r_2^n x)\to h_i(x_0)$ uniformly, the vanishing of the terms with $e^{w_1^n}$ and $e^{-c_2^n}$, and the $L^2$-smallness of the $f_n$-term), we obtain
\begin{equation}\label{sp2.1}
-\Delta w_2=\rho_2\,h_2(x_0)\,e^{w_2}\quad\text{ in } \quad\mathbb{R}^2.
\end{equation}
\par

We distinguish two cases according to the value of $h_2(x_0)$.

\noindent\emph{Case (i)} If $h_2(x_0)=0,$ then $w_2$ is harmonic in $\mathbb{R}^2$. Since $e^{w_2}$ is subharmonic and $e^{w_2}(0)=1$, the mean-value inequality gives $\int_{B_R(0)} e^{w_2}\ge \pi R^2$ for all $R>0$, hence $\int_{\mathbb{R}^2} e^{w_2}=\infty$.
This contradicts Lemma \ref{Bounded-F} because
\begin{equation}\label{sp3}
\begin{aligned}
\int_{\R^2}e^{w_2}=\lim_{n\rightarrow\infty}\int_{B_{\tilde{r}/r_2^n}(0)}e^{\psi(x^n_2+r^n_2x)+w^n_2}\leq
\lim_{n\rightarrow\infty}\int_{\Sigma}e^{u^n_2}dV_g&=\lim_{n\rightarrow\infty}{\frac{\int_{\Sigma}e^{-u_n}dV_g}{\int_{\Sigma}h_2e^{-u_n}dV_g}} \le C.
\end{aligned}
\end{equation}

\noindent\emph{Case (ii)} If $h_2(x_0)>0,$ then from \eqref{sp2.1} we have
\begin{equation*}
\rho_2 h_2(x_0)\int_{B_R(0)}e^{w_2}=\rho_2 \lim_{n\rightarrow\infty}\int_{B_R(0)}(h_2e^{\psi})(x_2^n+r_2^n x)e^{w^n_2(x)}dx \le \lim_{n\rightarrow\infty}\int_\Sigma \rho_2 h_2e^{u^n_2}dV_g=\rho_2.
\end{equation*}
By the classification in \cite{ChenLi}, $\rho_2 h_2(x_0) \int_{\R^2} e^{w_2}=8\pi$, which contradicts the  above inequality for $R$ sufficiently large.
Thus, we have $r_2^n/r_1^n \to \infty$. Finally, $c_1^n \to \infty$ since $c_1^n \ge c_2^n$ and $c_1^n+c_2^n \to \infty$.

\medskip

Now we prove that $h_1(x_0)>0$ where $x_0=\lim\limits_{n\to\infty} x_1^n$. Work in an isothermal coordinate system near $x_0$ (we will use the notation $g=e^{\psi(x)}|dx|^2$ again) and define 
\begin{equation*}
v_i^n(x):=u_i^n(x_1^n+r_1^n x)+2\log r_1^n, \qquad x\in B_{\tilde{r}/r_1^n}(0) \subset \R^2.
\end{equation*}
Then $v_1^n(0)=0$, and arguing as before, we obtain the analogue of \eqref{sp1} for $v_1^n$, and $v^n_1 \to v_1$ weakly in $H^2_{\rm{loc}}(\mathbb{R}^2)$ and strongly in $C^{\alpha}_{\rm{loc}}(\mathbb{R}^2)$.
Moreover, as in \eqref{sp2} (with $r_2^n$ replaced by $r_1^n$), we have $v_2^n \to -\infty$ locally uniformly in $\R^2$. Taking the limit in the rescaled equation gives 
\begin{equation}\label{sp3.1}
-\Delta v_1=8\pi h_1(x_0)e^{v_1} \quad \text{in } \R^2.
\end{equation}
If $h_1(x_0)=0,$ then $v_1$ is harmonic in $\mathbb{R}^2$, and using the same argument as in \eqref{sp3} we derive a contradiction with Lemma \ref{Bounded-F}.
Thus, $h_1(x_0)>0$, and it follows from the classification result \cite{ChenLi} that
\begin{equation}\label{sp3.2}
v_1(x)=-2\log\big(1+\pi h_1(x_0)|x|^2\big).
\end{equation}

\vspace{0.3cm}

\noindent \textbf{Proof of (2)}  Now we prove \eqref{selection1}. 
We work in an isothermal coordinate system near $x_0=0 \in \R^2$ such that $x_0$ is the unique blow-up point in $B_{3\tilde{r}}(0)$ for some $\tilde{r}>0$.
By \eqref{sp3.2}, there exists $R_n \to \infty$ such that
\begin{equation*}
v^n_1(y)+2\log|y|\leq C, \quad \forall y \in B_{R_n}(0).
\end{equation*}
With a change of variables $x=x_1^n+r_1^n y$, we can find $l^n_1\rightarrow 0$ such that $l_1^n/r_1^n \to \infty$ and
\begin{equation}\label{sp4.1}
u^n_1(x)+2\log|x-x^n_1|\leq C,  \quad \forall x\in B_{l_1^n}(x_1^n).
\end{equation}

Assume by contradiction that
\begin{equation}\label{sp4.2}
\max\limits_{i=1,2, |x|\leq \tilde{r}}\big(u^n_i(x)+2\log|x-x^n_1| \big) \rightarrow \infty.
\end{equation}
Let $q_n \in \overline{B}_{\tilde{r}}(0)$ be a point where the above maximum is attained, and define, for $i=1,2$,
\begin{equation*}
d_n:={\tfrac{1}{2}}|q_n-x^n_1|, \qquad S^n_i(x):=u^n_i(x)+2\log\big(d_n-|x-q_n|\big)\ \ \text{in }\ B_{d_n}(q_n).
\end{equation*}
Then $S_i^n(x) \to -\infty$ as $x\to \partial B_{d_n}(q_n)$, while, as $n\to\infty$, it follows from \eqref{sp4.2} that 
\begin{equation*}
\max_{i=1,2} S^n_i(q_n) =\max_{i=1,2}\big(u_i^n(q_n)+2\log d_n \big) \ge \max_{i=1,2}\big(u_i^n(q_n)+2\log|q_n-x_1^n|\big)-2\log 2 \rightarrow\infty.
\end{equation*}
Let $p_n$ be the point where $\max\limits_{x\in\overline{B}_{d_n}(q_n)}\{S^n_1,S^n_2\}$ is attained.
We distinguish two cases comparing $S_1^n(p_n)$ and $S_2^n(p_n)$.

\medskip

\noindent\emph{Case (i)}
Assume $S^n_1(p_n) \ge S^n_2(p_n)$. Then
\begin{equation}\label{sp5}
u^n_1(p_n)+2\log\big(d_n-|p_n-q_n|\big)=S_1^n(p_n)\geq \max\{S^n_1(q_n),\ S^n_2(q_n)\}\rightarrow+\infty.
\end{equation}
Let $l_n={\tfrac{1}{2}}\big(d_n-|p_n-q_n|\big)$. For any $y \in B_{l_n}(p_n)$ and $i=1,2$,
\begin{equation}\label{sp6}
\begin{aligned}
&u^n_i(y)+2\log\big(d_n-|y-q_n|\big)\leq u^n_1(p_n)+2\log(2l_n),\\
&d_n-|y-q_n|\geq d_n-|p_n-q_n|-|y-p_n|\geq l_n,
\end{aligned}
\end{equation}
hence
\begin{equation*}
u^n_i(y)\leq u^n_1(p_n)+2\log2, \quad \text{for all } y\in B_{l_n}(p_n), \ \ i=1,2.
\end{equation*}

Define the rescaled functions $\tilde{u}^n_i(z):=u^n_i(p_n+r_nz)+2\log r_n$, $i=1,2$, where $r_n:=e^{-u^n_1(p_n)/2}.$
Then, by \eqref{sp5} and \eqref{sp6},
\begin{equation*}
\tilde{u}_i^n(z) \le 2\log2, \quad \text{for all } |z| \le l_n/r_n, \ i=1,2, \quad \text{ and } \quad r_n \to 0, \ \ l_n/r_n \to \infty.
\end{equation*}
Moreover, $\tilde{u}^n_1$ satisfies on $B_{l_n/r_n}(0)$ the rescaled equation analogous to \eqref{sp1}. By $L^p$-estimates and Sobolev embedding, we obtain that 
$\tilde{u}^n_1 \rightharpoonup \tilde{u}_1$ in $H^2_{loc}(\R^2)$ and $\tilde{u}^n_1 \to \tilde{u}_1$ in $C^\alpha_{loc}(\R^2)$.
Arguing as in \eqref{sp2}, we have $\tilde{u}^n_1+\tilde{u}^n_2 \to -\infty$ locally in $\R^2$, hence $\tilde{u}^n_2 \to -\infty$ locally uniformly in $\R^2$.
Taking the limit, we obtain
\begin{equation}\label{sp7}
-\Delta \tilde{u}_1=8\pi h_1(x_0){e^{\tilde{u}_1}}\ \ \text{ in }\ \mathbb{R}^2,
\end{equation}
since $u_1^n(p_n) \to \infty$ and $x_0=0$ is the unique blow-up point in $B_{\tilde{r}}(0)$.
Hence, by \eqref{sp7} and the classification result \cite{ChenLi}, 
\begin{equation*}
1=\beta_1\int_{\mathbb{R}^2}e^{\tilde{u}_1}dy=\lim_{n\rightarrow\infty}\int_{B_{l_n/ 2r_n}(0)} (h_1e^{\psi})(p_n+r_ny)e^{\tilde{u}^n_1(y)}dy = \lim_{n\rightarrow\infty}\int_{B_{l_n/2}(p_n)}h_1e^{u^n_1+\psi}dy.
\end{equation*}
\par

On the other hand, from \eqref{sp3.2} (the blow-up at $x_1^n$) we have
\begin{equation*}
1=h_1(x_0)\int_{\mathbb{R}^2}e^{v_1}dy=\lim_{n\rightarrow\infty}\int_{B_{l^n_1/2r^n_1}(0)}(h_1e^{\psi})(x^n_1+r^n_1y)e^{v^n_1(y)}dy = \lim_{n\rightarrow\infty}\int_{B_{l^n_1/2}(x^n_1)}h_1e^{u^n_1+\psi}dy,
\end{equation*}
where we used $l_1^n/2r_1^n \to \infty$.

Since $u_1^n(p_n)+2\log|p_n-x_1^n| \ge u_1^n(p_n)+2\log(2l_n)-C \to \infty$, the inequality \eqref{sp4.1} implies $p_n \notin B_{l_1^n(x_1^n)}$. 
Moreover, by the definition of $d_n, l_n$, we also have $x_1^n \notin B_{l_n}(p_n)$. Therefore, for large $n$, $B_{l^n_1/2}(x^n_1) \cap B_{l_n/2}(p_n) =\emptyset$, so combining the above integration identities derives a contradiction to the fact that $\int_\Sigma h_1 e^{u_1^n}dV_g=1$. Hence, Case (i) cannot occur.

\medskip

\noindent\emph{Case (ii)}
 Suppose that\ $S^n_2(p_n) \ge S^n_1(p_n)$.  Then
we define $\tilde{u}^n_i(y):=u^n_i(p_n+r_ny)+2\log r_n$, $i=1,2$, where $r_n:=e^{-u^n_2(p_n)/2}.$
As in Case (i), by $L^p$-estimates, we have that $\tilde{u}^n_2 \rightharpoonup \tilde{u}_2$ in $H^2_{loc}(\R^2)$, $\tilde{u}^n_2 \to \tilde{u}_2$ in $C^\alpha_{loc}(\R^2)$,
and $\tilde{u}_2$ satisfies
\begin{equation*}
-\Delta \tilde{u}_2=\rho_2\beta_2{e^{\tilde{u}_2}}\ \ \text{in}\ \mathbb{R}^2, \ \  \text{ where } \beta_2=\lim_{n\rightarrow\infty}h_2(p_n).
\end{equation*}
By the same argument used for \eqref{sp2.1}, we derive a contradiction. Hence, Case (ii) cannot occur.
\vspace{0.3cm}

Combining the above, we obtain \eqref{selection1} on $B_{\tilde r}(x_0)$. 
If there were another blow-up point $\tilde{x}_0 \ne x_0$, then the same blow-up analysis at $\tilde x_0$ yields, for every $\delta>0$,
\begin{equation*}
\int_{B_\delta^g(\tilde x_0)} h_1 e^{u_1^n}\,dV_g \to 1
\quad\text{or}\quad
\int_{B_\delta^g(\tilde x_0)} \rho_2 h_2 e^{u_2^n}\,dV_g \to 8\pi.
\end{equation*}
However, we already have $\int_{B_\delta^g(x_0)} h_1 e^{u_1^n}\,dV_g \to 1$, while 
$\int_{\Sigma} h_1 e^{u_1^n}\,dV_g=1$ and $\int_{\Sigma}\rho_2 h_2 e^{u_2^n}\,dV_g=\rho_2<8\pi$, a contradiction. 
Hence $x_0$ is the unique blow-up point.

Consequently, $u_1^n$ and $u_2^n$ are uniformly bounded above on $\Sigma\setminus B_{\tilde r}(x_0)$; This proves \eqref{selection1} and completes the proof of Proposition~\ref{Selection}.
\end{proof}

\begin{remark}\label{rmk3.5}
From the proof of Proposition \ref{Selection}, we also obtain a useful consequence:
In contrast to the blow-up behavior $u_1^n(x_1^n) \to \infty$, it holds that
\[\overline{u}_1^n \rightarrow-\infty, \ \ \text{ and } \ \ \mu_1=8\pi\delta_{x_0},\]
where $\delta_{x_0}$ is the Dirac measure concentrated at $x_0$.
Consequently, Lemma \ref{Bounded} implies that, for any compact subset $K\Subset\Sigma\setminus \{x_0\}$, $u_1^n \rightarrow-\infty$ uniformly on $K$.

In fact, using \eqref{sp3.1} and the classification result \cite{ChenLi} (see \eqref{sp3.2}) in the proof of Proposition \ref{Selection}, we further deduce that, for any $\delta>0$,
\begin{equation*}
8\pi=8\pi\int_{\R^2} h_1(x_0)e^{v_1}dx=8\pi\lim_{n\to\infty} \int_{B_\delta^g(x_0)}h_1 e^{u_1^n} dV_g =\mu_1(B_\delta^g(x_0)) \le \mu_1(\Sigma)=8\pi.
\end{equation*}
Hence $\mu_1(B_\delta^g(x_0))=8\pi$ for all sufficiently small $\delta$, which implies that $\mu_1=8\pi\delta_{x_0}$.

To see $\overline{u}_1^n \to -\infty$, fix $\delta>0$ so small that $h_1(x)\ge \epsilon>0$ on the annulus $A_\delta:=B_\delta^g(x_0) \setminus B_{\delta/2}^g(x_0)$.
By Lemma \ref{Bounded}, there exists $C_\delta>0$ such that
\begin{equation*}
e^{\overline{u}_1^n}\int_{A_\delta} h_1  dV_g \le C_{\delta}\int_{A_\delta} h_1 e^{u_1^n} dV_g \to 0 \quad  \text{as}\ n\to\infty.
\end{equation*}
Since $\int_{A_\delta} h_1 dV_g>0$, we conclude that $\overline{u}_1^n \to -\infty$. $\Box$
\end{remark}

Next, we focus on the second component $u_2^n$. Recall that $x_i^n$ is the maximum point of $u_i^n$, with $c_i^n = u_i^n(x_i^n)$, $r_i^n = e^{-c_i^n/2}$.
We work in an isothermal coordinate around the unique blow-up point $x_0=0$ (so $g=e^{\psi(x)}|dx|^2$ with $\psi(0)=0$), and set $s_n=|x_1^n-x_2^n|$.
Rescaling at the scale $s_n$ around $x_2^n$, define
\begin{equation*}
w_2^n(x):=u_2^n(x_2^n+s_n x)+2\log r_2^n, \quad
w_1^n(x):=u_1^n(x_2^n+s_n x)+2\log s_n,
\end{equation*}
Set also the unit vector $\tilde x_n:=(x_1^n-x_2^n)/s_n\in \partial B_1(0)$ (so $w_2^n(0)=0=\sup w_2^n$).
After passing to a subsequence if necessary, assume $\tilde x_n\to\tilde x_0\in \partial B_1(0)$.

\begin{prop}\label{Never-S}
The sequence $u^n_2$ is uniformly bounded above on $\Sigma$, i.e.   there exists $C>0$ such that $u_2^n(x)\le C$ for all $x\in\Sigma$ and all $n\in\mathbb{N}$.
\end{prop}
\begin{proof}
Suppose that $u_2^n$ also blows up, i.e. $r_2^n \to 0$ as $n\to\infty$.
Substituting $x=x^n_2$ into the inequality in \eqref{selection1} for $u^n_2$, we obtain $|x_1^n-x_2^n|/r_2^n\le C$ for some $C$ independent of $n$.
Taking a subsequence if necessary, we may assume $s_n/r^n_2 \to A \in [0,\infty)$ as $n \to \infty$.

Before analyzing the functions $w_i^n$, $i=1,2$, we claim that $s_n/r^n_1 \to \infty$ as $n \to \infty$.
To see this, suppose not. Then, up to a subsequence, $(x^n_2-x^n_1)/r^n_1 \to z_0 \in\mathbb{R}^2$.
As in the proof of Proposition \ref{Selection} (see \eqref{sp3.2}), we have 
\begin{equation*}
	v_1^n(x)=u^n_1(x^n_1+r^n_1x)+2\log r^n_1\rightarrow v_1(x)=-2\log\big(1+\pi h_1(x_0)|x|^2\big) \quad \text{ in } C^{\alpha}_{\rm{loc}}(\mathbb{R}^2).
\end{equation*} 
Since $v_1(x)$ is radially symmetric and strictly decreasing, $u_2^n(x_1^n+r_1^n x)=-v_1^n(x)+\text{const}$, the functions
$u^n_2(x^n_1+r^n_1x)$ cannot have a maximum at $x=(x^n_2-x^n_1)/r^n_1$ for large $n$.
This contradicts the definition of $x^n_2$. Hence, the claim holds.

\medskip

\noindent \textbf{Step 1. Asymptotic behavior of $w_1^n$.}
We claim that $w^n_1 \rightarrow-\infty$ locally uniformly in $\mathbb{R}^2 \setminus \{\tilde{x}_0\}$.

Suppose not. Then there exist $D,\delta>0$ such that $\max\limits_{x \in B_D(\tilde{x}_0)\setminus B_\delta(\tilde{x}_0)} w_1^n \ge c$ for all $n \in \mathbb{N}$. 
On the other hand, by the definition of $\tilde{x}_n=(x_1^n-x_2^n)/s_n$, we have $w_1^n(x) \le C-2\log| x-\tilde{x}_n|$.
Since $\tilde{x}_n \to \tilde{x}_0$, for large $n$, we have $\max\limits_{x \in B_{2D}(\tilde{x}_0)\setminus B_{\delta/2}(\tilde{x}_0)} w_1^n \le C.$
Hence, by Harnack type inequalities (\cite[Theorems 9.20, 9.22]{G-T}) applied on a compact subset $K \Subset B_D(\tilde{x}_0) \setminus B_{\delta}(\tilde{x}_0),$ there exists $C$ independent of $n$ such that $\min\limits_{x \in B_D(\tilde{x}_0)\setminus B_\delta(\tilde{x}_0)} w_1^n \ge -C$.
In particular, by the fact $h_1(x_0)>0$, there exists $\epsilon_0>0$ independent of $n$ such that
\begin{equation}\label{cs3.3.1*}
\int_{K} (h_1 e^{\psi})(x^n_2+s_n x)e^{w^n_1(x)}dx \geq 2\epsilon_0>0.
\end{equation}

Next, by the classification result (see \eqref{sp3.1}--\eqref{sp3.2}), there exists $R>0$ such that $1-\epsilon_0=h_1(x_0) \int_{B_R(0)} e^{v_1}dx$. 
By a change of variables, it follows that
\begin{equation}\label{cs3.3.2}
1-\epsilon_0=\lim_{n\rightarrow\infty}\int_{B_R (0)} (h_1 e^{\psi})(x^n_1+r^n_1x)e^{v^n_1(x)}dx\\
=\lim_{n\rightarrow\infty}\int_{B_{R r_1^n/s_n}(\tilde{x}_n)} (h_1 e^{\psi})(x^n_2+s_n x)e^{w^n_1(x)}dx.
\end{equation}
Since $s_n/r^n_1 \to \infty$ as $n \to \infty$, the sets $B_{Rr_1^n/s_n}(\tilde{x}_n)$ and $K$ are disjoint for large $n$. Combining \eqref{cs3.3.1*} and \eqref{cs3.3.2}, we then obtain that, for large $n$, 
\begin{equation}\label{cs3.3.3}
1=\int_{\Sigma} h_1e^{u_1^n}dV_g \ge \int_{B_{Rr_1^n/s_n}(\tilde{x}_n)\cup K} (h_1 e^{\psi})(x^n_2+s_n x)e^{w^n_1(x)}dx \ge 1+\epsilon_0>1.
\end{equation}
It is a contradiction. This completes the proof of the claim.

\medskip

\noindent \textbf{Step 2. Analysis of the PDE for $w_2^n$.} 
First, by the definition, $w_2^n(x) \le 0$ and $w_2^n(0)=0$. 
We also note that, in the isothermal coordinate around $x_0=0$, the rescaled function $w_2^n$ satisfies, in $B_{\tilde{r}/s_n}(0)\subset\mathbb{R}^2$
\begin{equation}\label{cs3.3}
\begin{aligned}
-\Delta w^n_2(x)=&\rho_2(h_2e^{\psi})(x^n_2+s_n x)e^{w^n_2(x)}\Big(\frac{s_n}{r^n_2}\Big)^2-8\pi (h_1e^{\psi})(x^n_2+s_n x)e^{w^n_1(x)}\\
&+(8\pi-\rho_2)e^{\psi(x^n_2+s_n x)}s_n^2+(f_n e^{\psi})(x^n_2+s_n x) e^{w_1^n(x)/2}s_n \|h_1e^{u_n}\|_{L^1(\Sigma)}^{\frac12}.
\end{aligned}
\end{equation}

Fix an annulus $A_{R,\delta}:=B_R(\tilde{x}_0)\setminus B_\delta(\tilde{x}_0)$ with $0<\delta<R<\infty$. 
By Step 1, $w_1^n\to -\infty$ locally uniformly on $\mathbb{R}^2\setminus\{\tilde x_0\}$; hence $(h_1 e^{\psi})(x_2^n+s_n x)\, e^{w_1^n(x)} \to 0$ uniformly on $A_{R,\delta}$.
Since $\|f_n\|_{L^2(\Sigma)} \to 0$ and $\int_\Sigma e^{u_n}=\int_\Sigma e^{u_0}$, we also have
\begin{equation*}\label{cs3.3.3*}
	\int_{B_R(\tilde{x}_0)\backslash B_{\delta}(\tilde{x}_0)}\Big| (f_n e^{\psi})(x^n_2+s_n x) e^{\frac12 w^n_1(x)} s_n \|h_1e^{u_n}\|_{L^1(\Sigma)}^{\frac12}\Big|^2
    \rightarrow 0.
\end{equation*}
Moreover, by $s_n/r_2^n \to A \in [0,\infty$), the remaining terms of the right-hand side of \eqref{cs3.3} are uniformly bounded in $L^2(A_{R,\delta})$.

By the Harnack type inequality and $L^p$-estimate, the sequence $\{w_2^n\}$ is bounded in $H^2_{loc}(\R^2 \setminus \{\tilde{x}_0\})$. After passing to a subsequence, we see that $w_2^n \rightharpoonup w_2$ weakly in $H^2_{loc}(\R^2 \setminus\{\tilde{x}_0\})$ and strongly in $C^\alpha_{loc}(\R^2 \setminus\{\tilde{x}_0\})$ for some $\alpha \in (0,1)$, where $w_2$ satisfies
\begin{equation*}
\Delta w_2 + A^2 \rho_2 h_2(x_0)e^{w_2}=0 \ \ \text{ in } \R^2 \setminus \{\tilde{x}_0\}.
\end{equation*}
Applying the $L^p$-estimates again, we obtain that $w_2^n \to w_2$ strongly in $H^2_{loc}(\R^2 \setminus\{\tilde{x}_0\})$.

\medskip 

It now remains to show that $w_2$ satisfies the equation in $B_1(\tilde{x}_0) \subset \R^2$. 
To see this, we decompose $w_2^n:=w_r^n+w_s^n$, where the singular part $w_s^n$ solves the Dirichlet problem
\begin{equation*}
\begin{cases}
\Delta w_s^n  =8\pi (h_1 e^{\psi})(x^n_2+s_n x)e^{w^n_1}-(f_n e^{\psi})(x^n_2+s_n x) e^{w_1^n/2} s_n \|h_1e^{u_n}\|_{L^1(\Sigma)}^{\frac12} \ &\text{ in } B_1(\tilde{x}_0),\\
 w_s^n =0 \ &\text{ on } \partial B_1(\tilde{x}_0),
\end{cases}
\end{equation*}

By Cauchy-Schwarz inequality and \eqref{eq4.2}, we have that
\begin{equation*}
\int_{B_1 (\tilde{x}_0)} f_n(x^n_2+s_n x)s_n e^{w_1^n(x)/2}dx \leq C \|f_n\|_{L^2(\Sigma)}(\int_{\Sigma}e^{u_0}dV_g)^{\frac12}\rightarrow 0.
\end{equation*}
Moreover, by Step 1, we have
\begin{equation*}
h_1(x_2^n+s_n x)e^{w^n_1(x)} \rightharpoonup \delta_{\tilde{x}_0},
\end{equation*}
in the sense of measures.

By potential estimates (e.g. \cite[Lemma 7.12]{G-T}), the sequence $\{w_s^n\}_{n\in\N}$ is bounded in $W^{1,p}(B_1(\tilde{x}_0))$ for any $1<p<2$.  Hence, taking to a subsequence, we may assume that $w_s^n \to w_s$ weakly in $W^{1,p}_0(B_1(\tilde{x}_0))$, and 
\begin{equation*}
\Delta w_s = 8\pi\delta_{\tilde{x}_0} \ \ \text{ in } \R^2
\end{equation*}
in the sense of distributions. Consequently, $w_s(x)=4\log|x-\tilde{x}_0|$.
					
\medskip
					
Now we characterize the regular part $w^n_r$, which satisfies the following Dirichlet problem:
\begin{equation*}\begin{cases}
\Delta w_r^n =-\rho_2(h_2e^{\psi})(x^n_2+s_n x)e^{w^n_2}(\frac{s_n}{r^n_2})^2-(8\pi-\rho_2)e^{\psi(x^n_2+s_n x)}s_n^2  \ &\text{ in } B_1(\tilde{x}_0),\\
w_r^n=w_2^n \ &\text{ on } \partial B_1(\tilde{x}_0).
\end{cases}
\end{equation*}                    
Since $w_2^n\le 0$ and $w_2^n \to w_2$ in $H^2_{loc}(\R^2 \setminus\{\tilde{x}_0\})$ near $\partial B_1(\tilde{x}_0)$, applying standard ellpitic estimates, we obtain that $w_r^n \to w_r$ in $H^2(B_1(\tilde{x}_0))$ and $w_r$ satisfies
\begin{equation*}
\Delta w_r + A^2 \rho_2 h_2(x_0) e^{w_2}=0 \ \text{ in } B_1(\tilde{x}_0), \quad w_r=w_2 \ \ \text{ on } \partial B_1(\tilde{x}_0).
\end{equation*}
Thus, $w_2=w_r+4\log|\cdot-\tilde{x}_0|$ satisfies the following equation in distribution sense
\begin{equation}\label{cs3.2.5}
\Delta w_2+A^2\rho_2h_2(x_0)e^{w_2}=8\pi\delta_{\tilde{x}_0}\ \ \text{in}\ \ \mathbb{R}^2.
\end{equation}

We distinguish two cases according to the limiting coefficient:
\medskip

\noindent \emph{Case (i)} When $A^2 h_2(x_0)=0$.
Then \eqref{cs3.2.5} reduces to $\Delta w_2= 8\pi\,\delta_{\tilde x_0}$, hence $w_2$ is harmonic on $\mathbb{R}^2\setminus\{\tilde x_0\}$.
By construction $w_2\le 0$ and $w_2(0)=\sup_{\mathbb{R}^2\setminus\{\tilde x_0\}} w_2=0$; hence by the strong maximum principle  $w_2\equiv 0$ on $\mathbb{R}^2\setminus\{\tilde x_0\}$, which contradicts $\Delta w_2=8\pi\,\delta_{\tilde x_0}$.

\medskip

\noindent \emph{Case (ii)} When $A^2 h_2(x_0)>0$. We rewrite \eqref{cs3.2.5} into
\begin{equation*}\label{cs3.2.10}
\left\{\begin{aligned}
&-\Delta w_r(x)=A^2 \rho_2h_2(x_0)|x-\tilde{x}_0|^4 e^{w_r(x)} \quad \text{ in } \R^2,\\
&\int_{\mathbb{R}^2} |x-\tilde x_0|^{4} e^{w_r}
=\int_{\mathbb{R}^2} e^{w_2} \le\lim\limits_{n\to\infty}(\frac{r^n_2}{s_n})^2\int_{\Sigma}e^{u^n_2}dV_g={\frac{1}{A^2}} \lim\limits_{n\to\infty}{\frac{\int_{\Sigma}e^{-u_n}dV_g}{\int_{\Sigma}h_2e^{-u_n}dV_g}} <\infty.	
\end{aligned}\right.
\end{equation*}                    
By the classification result in \cite{PT}, it holds that $A^2 \rho_2h_2(x_0)\int_{\mathbb{R}^2}|x-\tilde{x}_0|^{4}e^{w_r}=24\pi$. Therefore, we can choose $R \gg 1$ and $0 <\delta \ll 1$ such that $A^2 \rho_2h_2(x_0)\int_{B_R(0) \setminus B_\delta(\tilde{x}_0)}|x-\tilde{x}_0|^{4}e^{w_r}>8\pi$. 
 However, this leads to a contradiction:
\begin{equation*}\label{cs3.2.11}
\begin{aligned}
8\pi<\rho_2 \lim_{n\to\infty}\frac{s_n^2}{(r_2^n)^2} \int_{B_R(0) \setminus B_\delta(\tilde{x}_0)}(h_2e^\psi)(x_2^n+s_n x)e^{w_2^n(x)} 
\leq\rho_2\int_{\Sigma}h_2e^{u^n_2}dV_g=\rho_2<8\pi.
\end{aligned}
\end{equation*}

\medskip
Both cases lead to contradictions, hence $u_2^n$ is uniformly bounded from above on $\Sigma$. 
This completes the proof of Proposition \ref{Never-S}.
\end{proof}

Now we can describe the global weak limit of the sequence $u_n-\overline{u}_n$. 
The next proposition identifies the limiting profile as a Green function plus a smooth correction, which will be the key input for the lower bound.

\begin{prop}\label{WeakConver}
For any $1<p<2$,  $u_n-\overline{u}_n+w_n$ converges to $G_{x_0}$ weakly in $W^{1,p}(\Sigma)$ and strongly in $W^{2,2}_{loc}\big(\Sigma\setminus\{x_0\}\big)$, where $G_{x_0}$ is the Green function in \eqref{green} with $p=x_0$,
and $w_n$ is the solution of the following equation
\begin{equation}\label{wc4.1*}
-\Delta_gw_n=\rho_2\Big({\frac{h_2e^{-u_n}}{\int_{\Sigma}h_2e^{-u_n}dV_g}}-1\Big) \ \  \text{ on } \ \  \Sigma, \quad \int_{\Sigma}w_ndV_g=0.
\end{equation}
In addition, up to a subsequence, $w_n\rightarrow w_{x_0}$ in $C^{1,\alpha}(\Sigma)$ for some constant $0<\alpha<1,$ where $w_{x_0}$ satisfies the singular mean field equation \eqref{smfeq} with $p=x_0$.
\end{prop}
\begin{proof}
Observe that $u_n-\overline{u}_n+w_n$ solves
\begin{equation}\label{wc4.2}
-\Delta_g\big(u_n-\overline{u}_n+w_n\big)=8\pi\big(\frac{h_1e^{u_n}}{\int_{\Sigma}h_1e^{u_n}dV_g}-1\big)-f_ne^{\frac{1}{2}u_n} \quad \text{ on } \quad \Sigma.
\end{equation}
By Remark \ref{rmk3.5} and \eqref{eq4.2}, we have $h_1e^{u_n}/\int_{\Sigma}h_1e^{u_n}dV_g\rightharpoonup\delta_{x_0}$ in the sense of measure, and 
$\|f_ne^{\frac{1}{2}u_n}\|_{L^1(\Sigma)}\leq\|f_n\|_{L^2(\Sigma)}(\int_{\Sigma}e^{u_0}dV_g)^{\frac12}\to0$, as $n\to\infty$. 
Consequently, by the potential estimates, $u_n-\overline{u}_n+w_n\rightharpoonup G_{x_0}$ weakly in $W^{1,p}(\Sigma)$ for any $p\in(1,2)$.

On the other hand, by Remark \ref{rmk3.5}, $u_n \to -\infty$ uniformly on any compact set $K \Subset \Sigma\setminus \{x_0\}$. Therefore, 
\begin{equation*}\label{wc4.3}
\begin{aligned}
\int_K\Big|8\pi {\frac{h_1e^{u_n}}{\int_{\Sigma}h_1e^{u_n}dV_g}}-f_ne^{\frac{1}{2}u_n}\Big|^2dV_g\leq C\|f_n\|^2_{L^2(\Sigma)}\to 0 \quad \text{as }\ n\to\infty.
\end{aligned}
\end{equation*}
This estimate, combined with standard elliptic regularity theory, yields the strong convergence in $W^{2,2}_{loc}(\Sigma\setminus\{x_0\}).$

Finally, Proposition \ref{Never-S} implies the boundedness of $e^{-u_n}/\int_{\Sigma}h_2e^{-u_n}dV_g$.
Applying elliptic regularity theory to \eqref{wc4.1*}, we conclude that $\{w_n\}_{n\in\N}$ is bounded in $W^{2,q}(\Sigma)$ for every $q \in (1,\infty)$. 
By the Sobolev embedding theorem, it follows that $w_n \to w_{x_0}$ in $C^{1,\alpha}(\Sigma)$ for some $\alpha\in(0,1).$ 
This completes the proof.
\end{proof}

\subsection{Energy lower bound in the blow-up regime}
We are ready to establish a lower bound for the energy functional $J_{\rho_2}(u_n)$ along the blow-up sequence $u_n$. 
This will be important in Section \ref{Testf-Sec}, where we select initial conditions that prevent blow-up and show that the flow converges to a solution of the stationary problem \eqref{maineq-s}.

Recall that, for $p\in \Sigma$, $\Gamma_p$ denotes the set of solutions to the singular mean field equation \eqref{smfeq}.

\begin{prop}\label{lowerbound}
Let $u_n\in H^2(\Sigma)$ be a blow-up sequence. Then 
\begin{equation}\label{LB5.1}
\begin{aligned}
\lim_{n\rightarrow\infty} J_{\rho_2}(u_n)\geq\inf\limits_{p\in\Sigma}\inf\limits_{w_p\in\Gamma_p}\left\{\tilde{J}_{p}(w_{p})-4\pi A(p) -8\pi\log h_1(p)\right\}-8\pi\log\pi-8\pi,
\end{aligned}
\end{equation}
where $G_{p}$ is the Green function in \eqref{green}, $A(p)$ is the regular part of $G_{p}$ defined in \eqref{Ap},
and $\tilde{J}_{p}$ is the functional \eqref{tildeJ} for the singular mean field equation \eqref{smfeq}.
\end{prop}
\begin{proof}
Recalling the definition of the energy functional $J_{\rho_2}$ and the definition of $u_1^n$ in \eqref{Def-u1u2}, we have
\begin{equation}\label{LBeq1}
J_{\rho_2}(u_n)=\frac{1}{2}\int_{\Sigma}|\nabla_g u_1^n|^2dV_g-\rho_2\log\int_{\Sigma}h_2e^{-u_n+\overline{u}_n}dV_g+8\pi\overline{u}_1^n.
\end{equation}


In order to compute the second term in \eqref{LBeq1}, we claim that a sequence $(-u_n+\overline{u}_n)$ is uniformly bounded above.
If not, then by \eqref{bdd4.1} and the fact that $S=\{x_0\}$, there exists a sequence of points ${y_n}\to x_0$, such that $-u_n(y_n)+\overline{u}_n\to \infty$. 
On the other hand, by Lemma \ref{Bounded-F} and Remark \ref{rmk3.5}, we know that $\overline{u}^n_1\to-\infty$, while $\int_{\Sigma}h_1e^{u_n}dV_g$ remains bounded. Hence, $\overline{u}_n=\overline{u}^n_1+\log\int_{\Sigma}h_1e^{u_n}dV_g\to-\infty$. 
This implies that $u_n(y_n)\to-\infty$, as $n\to\infty$.
However, by Proposition \ref{Selection}, $x_0$ is the unique blow-up point of $u^n_1$. 
Since $\int_{\Sigma}h_1e^{u_n}dV_g$ is bounded, we conclude that $u_n(y_n)\to\infty$ as $n\to\infty$.
The two conclusions contradict each other, and therefore the claim is proved.


										
With this claim, we can compute the second term. By Proposition \ref{WeakConver}, we know that $u_n-\overline{u}_n+w_n \to G_{x_0}$ in $C^\alpha_{loc}(\Sigma\setminus\{x_0\})$, and $w_n \to w_{x_0}$ in $C^\alpha(\Sigma)$. Hence, for sufficiently small $\delta>0$, we obtain
\begin{equation}\label{LBeq3}
	\lim_{n\to\infty}\int_\Sigma h_2 e^{-u_n+\overline{u}_n}dV_g =\int_{\Sigma \setminus B_\delta^g(x_0)} h_2 e^{-G_{x_0}+w_{x_0}}dV_g +o_\delta(1) = \int_{\Sigma} h_2 e^{-G_{x_0}+w_{x_0}}dV_g +o_\delta(1).
\end{equation}

We next compute the first term in \eqref{LBeq1}. 
To this end, fix an isothermal chart $\Psi:B_{\tilde r}(0)\to\Sigma$ with $\Psi(0)=x_0$ and $g=e^{\psi(x)}|dx|^2$.
Set $x_n:=x_1^n$ and $r_n:=r_1^n$ as in \eqref{defxn}, and choose $0<\delta<\tilde r<Rr_n^{-1}\tilde r$.
All balls are taken in the chart and then pushed forward by $\Psi$; for brevity, in the remainder of the proof, we write $B_r(x_n)$ both for the Euclidean ball $B_r(\Psi^{-1}(x_n))\subset\mathbb{R}^2$ and for its image $\Psi(B_r(\Psi^{-1}(x_n)))\subset\Sigma$.
With this convention, we decompose $\Sigma= B_{R r_n}(x_n) \cup \big(B_\delta(x_n)\setminus B_{R r_n}(x_n)\big) \cup \big(\Sigma\setminus B_\delta(x_n)\big).$


Near the blow-up point $x_0$, considering the rescaled solution $v_1^n$ in Proposition \ref{Selection}, we obtain that
\begin{equation}\label{LB5.2}
\begin{aligned}
\lim_{n\rightarrow\infty}{\frac{1}{2}}\int_{B_{Rr_n}(x_n)}|\nabla_g u_1^n|^2dV_g&={\frac{1}{2}}\lim_{n\rightarrow\infty}\int_{B_R(0)}|\nabla v_1^n|^2 dx
=\pi\int^R_0\Big|{\frac{4\pi h_1(x_0) r}{\pi h_1(x_0) r^2+1}}\Big|^2rdr\\
&=8\pi\log\pi-8\pi+16\pi\log R+8\pi\log h_1(x_0)+o_R(1),
\end{aligned}
\end{equation}
where $o_R(1) \to 0$ as $R \to\infty$.

For the domain $\Sigma \setminus B^g_\delta(x_n)$, i.e. away from the blow-up point $x_0$, Proposition \ref{WeakConver} together with the identity $\int_\Sigma \nabla_g G_{x_0} \nabla_g w_{x_0} dV_g=8\pi w_{x_0}(x_0)$ yields 
\begin{equation}\label{LB5.3}
\begin{aligned}
&\lim_{n\rightarrow\infty}{\frac{1}{2}}\int_{\Sigma\setminus B_{\delta}(x_n)}|\nabla_gu_1^n|^2dV_g=\lim_{n\rightarrow\infty}{\frac{1}{2}}\int_{\Sigma\setminus B_{\delta}(x_0)}|\nabla_g\big(u_n-\overline{u}_n+w_n\big)-\nabla_gw_n|^2dV_g\\
&={\frac{1}{2}}\int_{\Sigma}|\nabla_gw_{x_0}|^2dV_g+{\frac{1}{2}}\int_{\Sigma\setminus B_{\delta}(x_0)} |\nabla_g G_{x_0}|^2 dV_g -8\pi w_{x_0}(x_0)+  o_\delta(1).
\end{aligned}
\end{equation}
Integrating by parts the second term of \eqref{LB5.3} and using the local expansion \eqref{Ap} (with $p=x_0$), we obtain
\begin{equation}\label{LB5.3*}\begin{aligned}
{\frac{1}{2}}\int_{\Sigma\setminus B_\delta(x_0)} |\nabla_g G_{x_0}|^2 dV_g&=-{\frac{1}{2}}\int_{\Sigma\setminus B_\delta(x_0)} (\Delta_g G_{x_0}) G_{x_0}dV_g  -{\frac{1}{2}}\int_{\partial B_\delta(x_0)} \frac{\partial G_{x_0}}{\partial \nu} G_{x_0}dS_g\\
&=-16\pi\log\delta+4\pi A(x_0) +o_\delta(1),
\end{aligned}\end{equation}
where $o_\delta(1) \to 0$ as $\delta \to 0$.

For the neck domain, we compare $u_n$ with harmonic functions (see \cite{LiZhu,SunZhu}).
Define the spherical mean $u_n^*$ of $u_n$ by
\begin{equation*}
u^*_n(r):={\frac{1}{2\pi}}\int^{2\pi}_0u_1^n(x_n+re^{i\theta})d\theta.
\end{equation*}
Then $u_1^n$ and $u_n^*$ satisfy the following inequality (e.g. see \cite[inequality (3.4)]{LiWang})
\begin{equation*}
\int_{B_s \setminus B_r} |\nabla u_n^*|^2 dx \le \int_{B_s \setminus B_r} \Big|\frac{\partial u_1^n}{\partial r}\Big|^2 dx.
\end{equation*}
Let $\tilde{u}_n^*$ be the harmonic function on the neck domain with boundary conditions $\tilde{u}_n^*(r)=u_n^*(r)$ for $r=\delta$ and $r=Rr_n$.
Then it satisfies the following inequality (e.g. see \cite[equation (31)]{LiZhu})
\begin{equation*}
{\frac{1}{2}}\int_{B^g_{\delta}(x_n)\setminus B^g_{Rr_n}(x_n)}|\nabla_g u_1^n|^2 dV_g
\ge {\frac{1}{2}}\int_{B^g_{\delta}(x_n)\setminus B^g_{Rr_n}(x_n)}|\nabla_g \tilde{u}_n^*|^2 dV_g ={\frac{{\pi}\big(u^*_n(\delta)-u^*_n(Rr_n)\big)^2}{\log\delta-\log(Rr_n)}}.
\end{equation*}
\par

Define $\tau_n:=u^*_n(\delta)-u^*_n(Rr_n)-\overline{u}_1^n-2\log r_n$  and $u_1^n-\overline{u}_1^n+w_n \to G_{x_0}$ in $C^\alpha_{loc}(\Sigma \setminus\{x_0\})$ from Proposition \ref{WeakConver}, we obtain that
\begin{equation}\label{LB5.4}
\begin{aligned}
\lim_{n\rightarrow\infty}(u^*_n(Rr_n)+2\log r_n)&=-2\log \big( \pi h_1(x_0)R^2\big)+o_R(1),\\
\lim_{n\rightarrow\infty}\big(u^*_n(\delta)-\overline{u}_1^n\big)&=-4\log\delta+A(x_0)-w_{x_0}(x_0)+o_{\delta}(1),
\end{aligned}
\end{equation}
hence
\begin{equation}\label{LB-Neck-2}
\tau_n \rightarrow4\log{\frac{R}{\delta}}+A(x_0)-w_{x_0}(x_0)+2\log\pi+2\log h_1(x_0)+o_{\delta}(1)+o_R(1) \quad \text{ as } n \to\infty.
\end{equation}
\par

Since $\overline{u}_1^n \to -\infty$ and $r_n \to 0$ as $n\to\infty$, by straightforward calculation (e.g. see \cite{LiZhu,SunZhu}), we obtain that for large $n$,
\begin{equation}\label{LB5.5}
\begin{aligned}
&{\frac{{\pi}\big(u^*_n(\delta)-u^*_n(Rr_n)\big)^2}{\log\delta-\log(Rr_n)}}={\frac{\pi\big(\tau_n+\overline{u}_1^n+2\log r_n\big)^2}{-\log r_n}}\Big(1-{\frac{\log R-\log\delta}{-\log r_n}}\Big)^{-1}\\
&\geq{\frac{\pi\big(\tau_n+\overline{u}_1^n+2\log r_n\big)^2}{-\log r_n}}\Big(1+{\frac{\log R-\log\delta}{-\log r_n}}+{\frac{A'}{(\log r_n)^2}}\Big)\\
&\geq-\pi\log r_n\big(2-{\frac{\overline{u}_1^n}{\log r_n}}\big)^2-8\pi\overline{u}_1^n -2\pi\tau_n(2+{\frac{\overline{u}_1^n}{\log r_n}})+\pi\big(2+{\frac{\overline{u}_1^n}{\log r_n}}\big)^2\log{\frac{R}{\delta}}\\
&\quad+{\frac{2\pi\overline{u}_1^n\tau_n}{(\log r_n)^2}}\log{\frac{R}{\delta}}+{\frac{\pi A'}{-\log r_n}}\big(2+{\frac{\overline{u}_1^n}{\log r_n}}\big)^2+{\frac{2\pi A'\tau_n}{-(\log r_n)^2}}\big(2+{\frac{\overline{u}_1^n}{\log r_n}}\big) \\
&\quad\quad+o_R(1)+o_{\delta}(1)+o_n(1).
\end{aligned}
\end{equation}
\par

Since $J_{\rho_2}(u(t))$ is decreasing with respect to $t$, letting $n\to\infty$, we must have  $\overline{u}_1^n/\log r_n \rightarrow2$.
Indeed, if this were not the case, then as $r_n \to 0$ and $\overline{u}_1^n \to -\infty$, the first term $-\pi \log r_n (2-{\frac{\overline{u}_1^n}{\log r_n}}\big)^2$ and the second term $-8\pi\overline{u}_1^n$ in \eqref{LB5.5} would dominate all the remaining terms.
Moreover, combining \eqref{LBeq1}--\eqref{LB5.5}, we obtain
\begin{equation*}
J_{\rho_2}(u_n) \ge -\pi\log r_n\big(2-{\frac{\overline{u}_1^n}{\log r_n}}\big)^2 (1+o_n(1)).
\end{equation*}
since the contribution of $8\pi\overline{u}_1^n$ in $J_{\rho_2}(u_n)$ is canceled by the second term on the right-hand side of \eqref{LB5.5}.
This contradicts the boundedness of the energy $J_{\rho_2}(u_n)$. Therefore, we conclude that $\overline{u}_1^n/\log r_n \rightarrow2$.
\par

Substituting $\lim\limits_{n\to\infty}\overline{u}_1^n/\log r_n=2$ and the expression of $\tau_n$ from \eqref{LB-Neck-2} into \eqref{LB5.5}, we derive the inequality over the neck region:
\begin{equation}\label{LB-Neck-5}
	\begin{aligned}
		{\frac{1}{2}}\int_{{B_{\delta}(x_n)\backslash {B_{Rr_n}(x_n)}}}|\nabla_gu^n_1|^2dV_g\geq&-16\pi\log r_n-8\pi A(x_0)+8\pi w_{x_0}(x_0)-16\pi\log\pi\\
		&\quad-16\pi\log h_1(x_0)-16\pi\log{\frac{R}{\delta}}+O(\frac{1}{\log r_n}).
	\end{aligned}
\end{equation}
\par

Finally, combining \eqref{LBeq1}--\eqref{LB5.3*} and \eqref{LB-Neck-5} together, and letting $n \to \infty$ first and $R\to \infty$, $\delta \to 0$ next, we obtain the desired lower bound \eqref{LB5.1}.
%
This completes the proof.
\end{proof}

\medskip
\noindent

\section{Proof of Theorem \ref{main-thm}}\label{Testf-Sec}

In this section, we prove Theorem \ref{main-thm}. Guided by the limit profile in Proposition \ref{WeakConver} and the energy lower bound \eqref{LB5.1}, we construct test profiles whose energy lies strictly below the barrier level. This provides initial data for which blow-up is precluded and the flow converges to a stationary solution.

\subsection{Construction of a sub-barrier test function}
Define the energy barrier level
\[
L_* := \inf_{p\in\Sigma} \inf_{w\in\Gamma_p}\Big\{\tilde J_p(w)-4\pi A(p)-8\pi\log h_1(p)\Big\}-8\pi\log\pi-8\pi .
\]
Thus $L_*$ is the barrier below which blow-up cannot occur by \eqref{LB5.1}. 
We first show that the infimum is attained by a minimizing pair with $p_0 \in \Sigma$ and $w_{p_0} \in \Gamma_{p_0}$.


\begin{prop}\label{Compt}
There exist $p_0 \in \Sigma$ and $w_{p_0} \in \Gamma_{p_0}$ such that
\begin{equation}\label{Compt-w-x-0}
\tilde{J}_{p_0}(p_0)-4\pi A(p_0) -8\pi\log h_1(p_0)=\inf_{p\in \Sigma}\,\inf_{w_{p}\in \Gamma_{p}} \left\{\tilde{J}_{p}(w_p)-4\pi A(p) -8\pi\log h_1(p)\right\}.
\end{equation}
\end{prop}

\begin{proof}
Let $(p_n,w_{p_n})$ be a minimizing sequence of $L_*$, and 
up to a subsequence, we may assume that $p_n \to p_0\in\Sigma$.

Define $\xi_n:=w_{p_n} -G_{p_n}-\log\int_{\Sigma}h_2\exp(-G_{p_n}+w_{p_n})dV_g$, and substitute $\xi_n$ into \eqref{smfeq}. Then $\xi_n$ satisfies a singular Liouville type equation
\begin{equation}\label{xi-1}
-\Delta_g \xi_n=\rho_2 h_2 e^{\xi_n}- \rho_2 - 8\pi(\delta_{p_n}-1) \ \ \text{ on }\ \Sigma, \quad
\int_{\Sigma} e^{\xi_n}dV_g< C.
\end{equation}
\par

We claim that the sequence $\{\xi_n\}_{n\in\mathbb{N}}$ is uniformly bounded above.
If the claim is true, consider the following equation that $w_{p_n}$ satisfies
\begin{equation*}
-\Delta_gw_{p_n}=\rho_2h_2e^{\xi_n}-\rho_2\ \text{on}\ \Sigma,\quad \int_{\Sigma}w_{p_n}dV_g=0.
\end{equation*}
\par

By standard elliptic estimates, we have that $\{w_{p_n}\}_{n\in\mathbb{N}}$ is uniformly bounded in $H^1(\Sigma)$. Moreover, by $L^p$-estimates, $\{w_{p_n}\}_{n\in\mathbb{N}}$ is uniformly bounded in $W^{2,p}(\Sigma)$ and, up to a subsequence,
$w_{p_n} \to w_{p_0}$ in $C^1(\Sigma)$.
Therefore, the pair $(p_0,w_{p_0})$ attains the infimum.

Now it only remains to prove the claim. We first note that, for each $n\in \N$, by standard elliptic estimates, $w_{p_n} \in W^{2,p}(\Sigma)$ and $\max\limits_{x\in\Sigma} w_{p_n}(x) < \infty$, and it follows that $\max\limits_{x\in\Sigma} \xi_n(x) < \infty$. 
The remaining point is to show that this upper bound can be chosen uniformly in $n$.
To the contrary, suppose that 
$\max\limits_{x\in\Sigma}\xi_n(x):=\xi_n(y_n) \rightarrow \infty$.
Taking a subsequence, we may assume $y_n \to y_0$.
Pick an isothermal coordinate system centered at $y_0$, such that $g=e^{\varphi}\big(dx^2_1+dx^2_2\big)$ and $\varphi(0)=1.$ 
Set $r_n:=\exp(-\xi_n(y_n)/2) \rightarrow 0$.

We distinguish two cases according to the relative position of $p_n$ and $y_n$ : 

\textbf{Case 1.} $|p_n-y_n|/r_n \to \infty$. Define the rescaled function $\psi_n(x):= \xi_n(r_n x+y_n)+2\log r_n$.
By standard blow-up analysis, we may assume that $\psi_n \to \psi$ weakly in $H^2_{loc}(\R^2)$ and strongly in $C^\alpha_{loc}(\R^2)$, and $\psi$ satisfies
\begin{equation*}
-\Delta_{\R^2} \psi(y)= \rho_2 h_2(y_0) e^{\psi(y)} \ \ \text{ in } \ \ \R^2, \quad \int_{\R^2} h_2(y_0)e^{\psi(y)}dy \le 1, \quad  \int_{\R^2} e^{\psi(y)}dy < \infty.
\end{equation*}
If $h_2(y_0)>0$, this contradicts the classification result \cite{ChenLi} since $\rho_2 <8\pi$.
If $h_2(y_0)=0$, then $\psi$ is harmonic, and it contradicts the fact that $\int_{\R^2} e^{\psi(y)}dy < \infty$.


\textbf{Case 2.} $|p_n-y_n|/r_n \to A \in [0,\infty)$. We define the rescaled function $\psi_n(x):= \xi_n\big(|p_n-y_n|x+y_n\big)+2\log r_n$.
By the arguments in Proposition \ref{Never-S}, we have that $\psi_n \to \psi$ weakly in $H^2_{loc}(\R^2 \setminus \{\overline{z}_0\})$ and strongly in $C^\alpha_{loc}(\R^2 \setminus\{\overline{z}_0\})$ where $\overline{z}_0:=\lim_{n\to\infty} (p_n-y_n)/|p_n-y_n| \in \R^2 \setminus \{0\}$.
Therefore, $\psi(x) \le 0$ and $\psi(0)=0$.
Moreover, by the elliptic regularity theory, $\psi$ satisfies
\begin{equation*}
-\Delta_{\R^2} \psi(y)= A^2 \rho_2 h_2(y_0) e^{\psi(y)} - 8\pi\delta_{\tilde{x}_0} \ \ \text{ in } \ \ \R^2, \quad \int_{\R^2} h_2(y_0)e^{\psi(y)}dy \le 1, \quad  \int_{\R^2} e^{\psi(y)}dy < \infty.
\end{equation*}
If $A^2 h_2(y_0)>0$, then this contradicts the classification result in \cite{PT}.
If $A^2 h_2(y_0)=0$, then it contradicts the maximum principle.

\medskip

All two cases lead to a contradiction. Thus, the sequence $\xi_n$ must be uniformly bounded above, and this completes the proof of the claim.
\end{proof}





\begin{remark}\label{px0}
For each $p\in\Sigma$, the functional
\[
\tilde J_p(u)=\tfrac12\int_\Sigma|\nabla_g u|^2\,dV_g-\rho_2\log\!\int_\Sigma h_2\,e^{-G_p}e^u\,dV_g
\]
admits a uniform lower bound, independent of $p$, i.e.
there exists $C>0$ such that $\tilde J_p(u)\ge -C$ for all $p\in\Sigma$ and all $u\in H^1(\Sigma)$ with $\int_\Sigma u=0$.
Indeed, in an isothermal coordinate centered at $p$, one has $e^{-G_p}\sim r^{4}\,\tilde h_p(x)$ with $\tilde h_p$ smooth and strictly positive; by compactness of $\Sigma$ and smooth dependence on $p$, the weighted singular Moser–Trudinger inequality (see \cite{CM}) holds with constants uniform in $p$.
In particular, any minimizer $w_p\in\Gamma_p$ satisfies $\tilde J_p(w_p)>-\infty$ uniformly in $p$.

By Proposition~\ref{Compt} there exists a minimizing pair $(p_0,w_{p_0})$ such that
\[
\tilde J_{p_0}(w_{p_0})-4\pi A(p_0)-8\pi\log h_1(p_0)
=\inf_{p\in\Sigma}\ \inf_{w\in\Gamma_p}\big\{\tilde J_p(w)-4\pi A(p)-8\pi\log h_1(p)\big\}<\infty.
\]
Hence $-8\pi\log h_1(p_0)<\infty$ and therefore $h_1(p_0)>0$.
\end{remark}

With the minimizing pair $(p_0,w_{p_0})$ from Proposition \ref{Compt} in hand, we construct a sub-barrier test function $\tilde{\Phi}_\epsilon$ with $J_{\rho_2}(\tilde{\Phi}_\epsilon)<L_*$. 
Unlike \cite{DJLW}, Proposition \ref{WeakConver} shows that $u_n-\overline{u}_n \rightharpoonup G_{x_0}-w_{x_0}$; motivated by this decomposition, we construction a function $\Phi_\epsilon$ centered at $p_0$ and subtract $w_{p_0}$.


We work in normal coordinates $(r,\theta)$ centered at $p_0$ and 
we will repeatedly use the standard expansions, uniform in $\theta$:
\begin{align}
G_{p_0}(x) &= -4\log r + A(p_0) + b_1 r\cos\theta + b_2 r\sin\theta + O(r^2), \label{eq:Green-exp}\\
dV_g &= \Big(1-\frac{K(p_0)}{6}r^2 + O(r^3)\Big) r dr d\theta. \label{eq:area-exp}
\end{align}

For $0<\epsilon \ll 1$, we define $\tilde{\Phi}_{\epsilon}:=\Phi_{\epsilon}-w_{p_0}$, where $\Phi_{\epsilon}$ constructed in \cite{DJLW} is given as follows:
\begin{equation}\label{tf6.1}
\Phi_{\epsilon}:=
\left\{\begin{aligned}
&-2\log\big(r^2+\epsilon\big)+b_1r\cos\theta+b_2r\sin\theta+\log\epsilon,\quad &x\in B_{\alpha{\sqrt{\epsilon}}}(p_0),\\
&\big(G_{p_0}-\eta\beta(r,\theta)\big)-2\log\Big(\tfrac{\alpha^2+1}{\alpha^2}\Big)-A(p_0)+\log\epsilon,\quad & x\in B_{2\alpha{\sqrt{\epsilon}}}(p_0)\backslash B_{\alpha{\sqrt{\epsilon}}}(p_0),\\
&G_{p_0}-2\log\Big(\tfrac{\alpha^2+1}{\alpha^2}\Big)-A(p_0)+\log\epsilon,\quad &x\in\Sigma\backslash B_{2\alpha{\sqrt{\epsilon}}}(p_0).
\end{aligned}\right.
\end{equation}
Here $\eta\in C^{\infty}_0(B_{2\alpha\sqrt{\epsilon}}(p_0))$ satisfies $\eta\equiv1$ in $B_{\alpha\sqrt{\epsilon}}(p_0)$ and $|\nabla_g\eta|\leq C/(\alpha\sqrt{\epsilon})$, and $\alpha=\alpha(\epsilon) \gg 1$ is chosen so that $\alpha^4\epsilon=1/\log(-\log\epsilon)\to0$ as $\epsilon \to 0$.

\begin{prop}\label{TFthm}
Suppose that $8\pi-2K(x)+\Delta_g\log h_1(x)-\rho_2>0$ for all $x\notin h^{-1}_1(\{0\})$. Then, for $\epsilon$ small enough, the following inequality holds:
\begin{equation}\label{tf6}
\begin{aligned}
J_{\rho_2}(\tilde{\Phi}_{\epsilon}) < \inf_{p\in\Sigma}\inf_{w_p\in\Gamma_{p} }\Big( \tilde{J}_{p}(w_p)-4\pi A(p) -8\pi\log h_1(p)\Big)-8\pi\log\pi-8\pi.
\end{aligned}
\end{equation}
\end{prop}
\begin{proof}
By substituting the test function $\tilde{\Phi}_\epsilon$ into the energy functional $J_{\rho_2}$ in \eqref{main-EF}, we obtain that
\begin{equation}\label{tf6.1*}
\begin{aligned}
J_{\rho_2}(\widetilde{\Phi}_{\epsilon})=&{\frac{1}{2}}\int_{\Sigma}|\nabla_g\Phi_\epsilon|^2dV_g+{\frac{1}{2}}\int_{\Sigma}|\nabla_g w_{p_0}|^2dV_g-\int_{\Sigma}\nabla_g\Phi_{\epsilon}\nabla_gw_{p_0}dV_g-8\pi\log\int_{\Sigma}h_1e^{(\Phi_\epsilon-w_{p_0})}dV_g\\
&-\rho_2\log\int_{\Sigma}h_2e^{-(\Phi_{\epsilon}-w_{p_0})}dV_g+(8\pi-\rho_2)\int_{\Sigma}\Phi_{\epsilon}dV_g.
\end{aligned}
\end{equation}

Using the expansion of $G_{p_0}(r,\theta)$ in \eqref{eq:Green-exp} together with the identities in \cite{DJLW}, the $\Phi_\epsilon$–only part can be computed as
\begin{equation}\label{tf6.2}
\begin{aligned}
&\frac12 \int_\Sigma |\nabla_g \Phi_\epsilon|^2 dV_g +8\pi \int_\Sigma \Phi_\epsilon dV_g -8\pi\log\int_{\Sigma}h_1e^{\Phi_{\epsilon}-w_{p_0}}dV_g\\
=&-8\pi-8\pi\log\pi-4\pi A(p_0)-8\pi \log (h_1e^{-w_{p_0}})(p_0)+16\pi^2 \Big( 1-\frac{K(p_0)}{4\pi}+\frac{b_1^2+b_2^2}{8\pi}\\
&\quad+\frac{\Delta_g (h_1 e^{-w_{p_0}})(p_0)}{8\pi h_1e^{-w_{p_0}}(p_0)}+\frac{(k_1b_1+k_2b_2)}{4\pi h_1 e^{-w_{p_0}}(p_0)}\Big)\cdot\epsilon (-\log\epsilon) + o(\epsilon(-\log\epsilon)),
\end{aligned}
\end{equation}
where $(k_1,k_2):=\nabla_g(h_1e^{-w_{p_0}})(p_0)$, and
\begin{equation}\label{tf6.3}
\begin{aligned}
\int_{\Sigma}\Phi_{\epsilon}dV_g=&\log\epsilon-2\pi\alpha^2\epsilon\log\big(\frac{\alpha^2+1}{\alpha^2}\big)-2\pi\epsilon\log\big(\alpha^2+1\big)-A(p_0)-2\log\big(\frac{\alpha^2+1}{\alpha^2}\big)\big(1-|B_{\alpha\sqrt{\epsilon}}|\big)\\
&\quad+O\big(\alpha^4\epsilon^2\log(\alpha^2\epsilon)\big).
\end{aligned}
\end{equation}

Moreover, we have

\begin{equation}\label{tf6.5**}
\begin{aligned}
\int_{\Sigma}\nabla_g\Phi_{\epsilon}\nabla_gw_{p_0}dV_g
=&\int_{\Sigma}\nabla_gG_{p_0}\nabla_g w_{p_0}dV_g+2\rho_2\pi\epsilon\log(\alpha^2+1)+O(\epsilon),
\end{aligned}
\end{equation}
\begin{equation}\label{tf6.6}
\begin{aligned}
\log\int_{\Sigma}h_2e^{-(\Phi_{\epsilon}-w_{p_0})}dV_g
&=-\log\epsilon+2\log(\frac{\alpha^2+1}{\alpha^2})+A(p_0)+\log\int_{\Sigma}h_2e^{-G_{p_0}+w_{p_0}}dV_g+O(\alpha^4\epsilon^3).
\end{aligned}
\end{equation}

Substituting \eqref{tf6.2}--\eqref{tf6.6} into \eqref{tf6.1*} and using $\int_\Sigma \nabla_g G_{p_0}\nabla_g w_{p_0}dV_g=8\pi w_{p_0}(p_0)$, we obtain
\begin{equation}\label{tf6.7}
\begin{aligned}
J_{\rho_2}(\tilde{\Phi}_{\epsilon})=&\Big(\tilde{J}_{p_0}(w_{p_0})-4\pi A(p_0)-8\pi\log h_1(p_0)\Big)-8\pi\log\pi-8\pi-16\pi^2\Big(1-{\frac{1}{4\pi}}K(p_0)\\
&\quad+{\frac{b^2_1+b^2_2}{8\pi}}+\frac{\Delta_g (h_1 e^{-w_{p_0}})(p_0)}{8\pi h_1e^{-w_{p_0}}(p_0)}+{\frac{k_1b_1+k_2b_2}{4\pi h_1e^{-w_{p_0}}(p_0)}}\Big)\cdot\epsilon\big(-\log\epsilon\big)+o\big(\epsilon(-\log\epsilon)\big).
\end{aligned}
\end{equation}

We now analyze the coefficient of $\epsilon(-\log\epsilon)$ in \eqref{tf6.7}. 
Using the fact $\Delta_gw_{p_0}(p_0)=\rho_2$ (from \eqref{smfeq}), we deduce that 
\begin{equation}\label{tf6.8}
\begin{aligned}
&-16\pi^2\Big(1-{\frac{1}{4\pi}}K(p_0)+{\frac{b^2_1+b^2_2}{8\pi}}+\frac{\Delta_g (h_1 e^{-w_{p_0}})(p_0)}{8\pi h_1e^{-w_{p_0}}(p_0)} + {\frac{k_1b_1+k_2b_2}{4\pi h_1e^{-w_{p_0}}(p_0)}}\Big)\\
=&-2\pi\Big(8\pi-2K(x)-\rho_2+\sum_{i=1}^2\big(b_1+{\frac{k_1}{h_1e^{-w_{p_0}}(x)}}\big)^2+{\frac{\Delta_g h_1}{h_1}}-{\frac{|\nabla_gh_1|^2}{h_1^2}}\Big)|_{x=p_0}\\
=&-2\pi\Big(8\pi-2K(p_0)+\Delta_g\log h_1(p_0)-\rho_2+ \sum_{i=1}^2\big(b_i+{\frac{k_i}{h_1e^{-w_{p_0}}(p_0)}}\big)^2\Big)<0.
\end{aligned}
\end{equation}
where the strict negativity follows from the assumption
$8\pi - 2K(x) + \Delta_g\log h_1(x) - \rho_2 > 0$ for all $x\notin h_1^{-1}(\{0\})$
since $h_1(p_0)>0$ by Remark \ref{px0}.

Therefore, substituting \eqref{tf6.8} into \eqref{tf6.7} and using Proposition \ref{Compt}, we conclude that 
\begin{equation*}
\begin{aligned}
 J_{\rho_2}(\tilde{\Phi}_{\epsilon})&<\Big(\tilde{J}_{p_0}(w_{p_0})-4\pi A(p_0)-8\pi\log h_1(p_0)\Big)-8\pi\log\pi-8\pi\\
 &=\inf\limits_{p\in\Sigma}\inf\limits_{w_p\in\Gamma_p}\left\{\tilde{J}_{p}(w_{p})-4\pi A(p)-8\pi\log h_1(p)\right\}-8\pi\log\pi-8\pi.
\end{aligned}
\end{equation*}
This completes the proof of Proposition \ref{TFthm}.
\end{proof}

\subsection{Convergence to a stationary solution}
In this subsection, we are now in position to complete the proof of our main theorem.  In particular, we shall prove the existence of a solution to \eqref{maineq-s}, provided the initial data $u_0$ is chosen suitably.


\medskip

\noindent \textbf{Proof of Theorem \ref{main-thm}.}
Let $\tilde{\Phi}_{\epsilon}$, as defined in Proposition \ref{TFthm}, be the initial datum of the flow. Using the monotonicity of the flow together with Proposition \ref{lowerbound} and Proposition \ref{TFthm}, the sequence $\{u_n\}$ does not blow up. By Lemma \ref{Bounded} together with standard elliptic estimates, we conclude that $\{\|u_n\|_{L^\infty(\Sigma)}\}$ is uniformly bounded. 
Hence, up to a subsequence, $u_n \to u_\infty$ weakly in $H^2(\Sigma)$ and strongly in $C^\alpha(\Sigma)$ as $n \to \infty$.
In particular, the limit $u_\infty$ satisfies the mean-field type equation \eqref{maineq-s}.

Now it remains to prove the convergence of the flow.
Using the estimates in Section 2, we first prove the boundedness of $\{\|u(t)\|_{C^{2,\alpha}(\Sigma)}\}$
Then, by a standard argument in parabolic theory, we show that $u(t) \to u_\infty$ in $L^2(\Sigma)$ sense.
Finally, applying the Arzela-Ascoli theorem yields that $u(t) \to u_\infty$ in $C^2(\Sigma)$.
Thus, the proof is completed once we establish the boundedness of $\{\|u(t)\|_{C^{2,\alpha}(\Sigma)}\}$ (Step 1) and the convergence $L^2(\Sigma)$ (Step 2).

\medskip

\noindent \textbf{Step 1.} To the contrary, we suppose that there exists a sequence $t_n \to \infty$ such that $\|u(t_n)\|_{C^{2,\alpha}(\Sigma)} \to \infty$.
For a fixed $T>0$, by \eqref{sec3eq1}, we can choose a sequence $s_n \to \infty$ such that
\begin{equation*}
t_n-T < s_n < t_n, \quad \text{ and } \quad \lim_{n \to \infty} \int_{\Sigma} \left|\frac{\partial u(s_n)}{\partial t}\right|^2e^{u(s_n)}dV_g = 0. 
\end{equation*}
\par

Since $J_{\rho_2}(u_0)$ is less than the lower bound in Proposition \ref{lowerbound}, the sequence $u_n:=u(s_n)$ does not blow up.
Applying the results in Section 3 to $u(s_n)$ (see Lemma \ref{Bounded} (2)), we obtain
\[u_1^n=u(s_n)-\log\int_\Sigma h_1 e^{u(s_n)}dV_g, \ \ u_2^n=-u_n-\log\int_\Sigma h_2 e^{-u(s_n)}dV_g \le C \ \  \text{ on } \ \ \Sigma,\]
for all $n \in \mathbb{N}$.
Moreover, since $\int_\Sigma h_1 e^{u(t)}dV_g \le C$ for all $t\ge 0$ (see Lemma \ref{Bounded-F}), applying the $L^p$-estimate for elliptic equations to \eqref{Original.Eq}, we obtain the boundedness of $\{\|u(s_n)\|_{H^2(\Sigma)}\}$.

By Proposition \ref{H2}, we deduce that $\|u(t)\|_{H^2(\Sigma)} \le C$ for all $t \in [s_n, s_n+2T)$ and $n\in \mathbb{N}$.
In addition, by the arguments in the proof of Proposition \ref{C2alpha}, we have
\[ \|u(x,t)\|_{C^{\alpha,\alpha/2}(\Sigma \times [s_n,s_n+2T))} \le C,\]
uniformly in $n\in \mathbb{N}$.
Then, by the Schauder estimates for parabolic equations, we have
\[ \|u(x,t)\|_{C^{2+\alpha,1+\alpha/2}(\Sigma \times [s_n,s_n+2T))} \le C,\]
and it contradicts the assumption for $t_n$.
This completes the proof of Step 1.

\medskip

\noindent\textbf{Step 2.} First, we observe that the energy functional $J_{\rho_2}:H^1(\Sigma) \to \R$ is analytic and $J_{\rho_2}'(H^2(\Sigma)) \subset L^2(\Sigma)$.
Moreover, for any critical point $u_\infty \in C^\infty(\Sigma)$ of $J_{\rho_2}$, the second derivative
$J_{\rho_2}''(u_\infty): H^1(\Sigma) \to H^{-1}(\Sigma)$ is a Fredholm operator with index 0.
By {\L}ojasiewicz-Simon gradient inequality (see \cite[Theorem 2]{PM}), there exist constants $Z\in (0,\infty)$, $\sigma\in(0,1]$ and $\theta\in[{\frac{1}{2}},1),$ such that for all $u\in H^2(\Sigma)$ with $\|u-u_{\infty}\|_{H^2(\Sigma)}<\sigma$,
\begin{equation*}\label{Gcs2.1}
Z\big|J_{\rho_2}(u)-J_{\rho_2}(u_{\infty})\big|^{\theta}
\leq \|J_{\rho_2}'(u)\|_{L^2(\Sigma)}
\end{equation*}
Since $u_n=u(t_n) \to u_\infty$, we can apply this inequality to the flow $u(t)$ for $t \in [t_n, T]$, where $\|u_n-u_\infty\|_{L^2(\Sigma)}  \ll \sigma$ and $T:=\inf\left\{t>t_{n}:\ \|u(t)-u_{\infty}\|_{L^2(\Sigma)}\geq\sigma\right\}$.
Then it follows that
\begin{equation*}\begin{aligned}\label{Gcs2.3}
-{\frac{d}{dt}}\big(J_{\rho_2}(u(t))-J_{\rho_2}(u_{\infty})\big)^{1-\theta} & =-(1-\theta)\big(J_{\rho_2}(u(t))-J_{\rho_2}(u_{\infty})\big)^{-\theta} \left\|e^{u(t)/2}\frac{\partial u(t)}{\partial t}\right\|_{L^2(\Sigma)}^2 \\
&\ge c(1-\theta)\left\|\frac{\partial u(t)}{\partial t}\right\|_{L^2(\Sigma)},
\end{aligned}\end{equation*}
for some $c>0$.
Consequently, for $s \in (t_n, T)$, we obtain
\begin{equation*}
\|u(s)-u(t_n)\|_{L^2(\Sigma)} \le \int_{t_n}^s \left\|\frac{\partial u(t)}{\partial t}\right\|_{L^2(\Sigma)} dt \le \frac{1}{c(1-\theta)}\big(J_{\rho_2}(u(s))-J_{\rho_2}(u_{\infty})\big)^{1-\theta}.
\end{equation*}

Choose $n$ sufficiently large so that $T=\infty$ and $\int_{t_n}^\infty \left\|\frac{\partial u(t)}{\partial t}\right\|_{L^2(\Sigma)} dt <\infty$.
This completes the proof of the convergence in $L^2(\Sigma)$.
 $\Box$

\end{document}